\documentclass[12pt,a4paper]{article}
\usepackage[utf8]{inputenc}
\usepackage{amsfonts, amsmath, amssymb, amsthm}
\usepackage{setspace,caption}
\usepackage[hidelinks]{hyperref}
\usepackage{titlesec}
\usepackage{graphicx}
\usepackage{enumitem}
\usepackage{cleveref}
\usepackage{textcomp}
\hypersetup{colorlinks={true},linkcolor={blue}}

\titleformat{\section}[block]{\bfseries}{\thesection.}{1em}{}
\titleformat{\subsection}[block]{\bfseries}{\thesubsection.}{1em}{}
\newtheorem{thm}{Theorem}
\newtheorem{cor}{Corollary}
\newtheorem{defi}{Definition}
\newtheorem{nt}{Notation}
\newtheorem{lem}{Lemma}
\newtheorem{rmk}{Remark}

\newcommand\floor[1]{\lfloor#1\rfloor}
\newcommand\genby[1]{\langle#1\rangle}

\hypersetup{urlcolor = blue}

\title{Irredundant generating sets and dimension-like invariants of the finite group}
\author{Minh Nguyen} 
\begin{document}
\maketitle
\begin{abstract}
Whiston proved that the maximum size of an irredundant generating set in the symmetric group $S_n$ is $n-1$, and Cameron and Cara characterized all irredundant generating sets of $S_n$ that achieve this size. Our goal is to extend their results. Using properties of transitive subgroups of the symmetric group, we are able to classify all irredundant generating sets with sizes $n-2$ in both $A_n$ and $S_n$. Next, based on this classification, we derive other interesting properties for the alternating group $A_n$. Finally, using Whiston's lemma, we will derive the formulas for calculating dimension-like invariants of some specific types of wreath products.
\end{abstract}
\newpage
\tableofcontents
\setlength{\parskip}{0pt}

\newpage
\section{Introduction}
This paper is the report for summer undergraduate research program at Cornell University. The paper mainly deals with an irredundant generating set and dimension-like invariants of a finite group. The paper consists of three main parts.

In the first part, we will classify all irredundant generating sets of length $n - 2$ in $A_n$. Then, from this classification, we will prove many other properties of the alternating group $A_n$ 

Cameron and Cara, in their paper, classify all irredundant generating sets of length $n - 1$ in $S_n$. In the second part of this paper, we will try to study a more general problem, which is the classification of all irredundant generating sets of arbitrary length in $S_n$.

In the final part, we will derive the formulas for calculating dimension-like invariants of some specific types of wreath products. 

We start with definitions of dimension-like invariants for a finite group.

\begin{defi}
For a finite group $G$, and a set $H \subset G$, let $\genby{H}$ denote the subgroup of $G$ generated by $H$.
 
For a finite group $G$, the set $s = (g_1, . . . , g_n)$ of elements of G is called irredundant if  $\genby{\{g_i\}_{i \neq j}} \neq \genby{\{g_i\}}\ \forall j \in \overline{1, n}$

Let $i(G)$ be the maximal size of an irredundant set in $G$, and $m(G)$ be the maximal size of an irredundant generating set in $G$. Then, clearly, $m(G) \leq i(G)$.

For $|G| = \prod\limits_{i =1}^kp_i^{\alpha_i}$, let $\lambda(G)$ denote the sum $\sum\limits_{i=1}^k \alpha_i$. Then it is easy to see that $i(G) \leq \lambda(G)$.

A finite group $G$ is flat if $m(H) \leq m(G)$ for all subgroup $H$ of $G$, and is strongly flat if $m(H) < m(G)$ for all proper subgroup $H$ of $G$. In both cases, $i(G) = m(G)$.
\end{defi}

\begin{nt}\label{nt1}
We denote the cycle in the symmetric group $S_n$ that maps $x_i$ to $x_{i + 1}$ for $i \in \overline{1, k}$, with $x_{k + 1} = x_1$, by $(x_1, x_2, \dotsc, x_k)$.
\end{nt}

\section{Properties of irredundant generating sets in the alternating group}

\begin{lem}(Whiston)\label{lm1}
Suppose $(g_1, \dotsc , g_m)$ is an irredundant generating set of some group $G$, and $N \trianglelefteq G$ is a normal subgroup of $G$. Then, possibly after reordering the $g_i$, there exists some $k \leq m$ and some elements $h_{k+1}, \dotsc , h_m \in N$ such that the projections $\overline{g_1}, \dotsc, \overline{g_k}$ form an irredundant generating set of $G/N$ and $g_1, \dotsc , g_k, h_{k+1}, \dotsc, h_m$ form a new irredundant generating set for $G$.
\end{lem}

\begin{defi}\label{def2}
Assume that $H$ is an imprimitive subgroup of $S_n$ or $H \leq S_{\Gamma} \wr S_{\Delta} \leq S_n$ with $n = |\Gamma||\Delta|$. Suppose also that $\Gamma_1,\dotsc, \Gamma_{|\Delta|}$ are copies of  $|\Gamma|$ or blocks such that we can consider $H$ as a group that acts on $n$ points of $\Gamma_1 \cup\dotsc \cup \Gamma_{|\Delta|}$ by permuting points of each individual block $\Gamma_i$, and by permuting blocks $\Gamma_i$ themselves. 

Let $\{g_1, g_2,\dotsc, g_l\}$ be an irredundant generating set for $H$. By \cref{lm1} for $G =S_{\Gamma} \wr S_{\Delta}$, and its normal subgroup $S_{\Gamma}^{|\Delta|} = \prod\limits_{i = 1}^{|\Delta|}S_{\Gamma_i}$, we can assume that for some $k$, $g_1,\dotsc, g_k$ generates block action for $H$, and $g_{k+1},\dotsc, g_l$ fix the blocks. 
\end{defi}

\begin{lem}\label{lm2}(Whiston)
Suppose $H$ is the subgroup of $S_n$ that is defined as in \cref{def2}. If the subgroup generated by $\{g_{k+1},\dotsc, g_l\}$ acts as $S_{\Gamma_1}$ or $A_{\Gamma_1}$ on $\Gamma_1$, then $m(H) = l \leq |\Gamma| + 2|\Delta| - 3$. Note that the lemma is true if $\Gamma_1$ is replaced by any $\Gamma_i$.
\end{lem}

\begin{thm}\label{thm1}(O'Nan-Scott's theorem. \hyperlink{Wil}{See references})
If $H$ is any proper primitive subgroup of $S_n$ other than $A_n$, then $H$ is a subgroup of one of the following subgroups:
\begin{enumerate}
\item A primitive wreath product, $S_k \wr S_m$, where $n = k^m$;
\item An affine group $AGL_s(p)$, where $n = p^s$
\item A group of shape  $T^k \cdot (Out(T) \times S_k )$, where $T$ is a non-abelian simple group, acting on the cosets of 'diagonal' subgroup of $Aut(T) \times S_k$, where $n = |T|^{k-1}$
\item $H \leq Aut(T)$ with some simple group $T \leq S_n$
\end{enumerate}
\end{thm}

\begin{lem}\label{lm3}
$H \neq A_n$ is a primitve subgroup of $S_n$, $n \geq 13$. Then $m(H) \leq n - 4$. More generally, there exists $n_0 \in \mathbb{Z}$ such that for any $n \geq n_0$, and any primitive subgroup $H$ of $S_n$, we have $m(H) \leq \frac{n}{2}$.
\end{lem}
\begin{proof}
We will follow Whiston's proof to prove for $m(H) \leq n - 4$. To prove the more general inequality, we just need to follow exactly every case Whiston considered in his proof.

Based on \hyperref[thm1]{O'Nan-Scott}'s theorem and the classification of finite simple groups, we will primitive subgroup of $S_n$ other than $A_n$ cases by cases to prove that $m(H) \leq n -4$:
\vspace{5mm}

$\mathbf{Case\ 1}$: $H$ is a subgroup of the affine group $AGL_s(p)$. Then $|H| \leq |AGL_s(p)| \leq p^{s^2+s}$, so $m(H) \leq \lambda(H) \leq \log_2|H| \leq (s^2+s)\log_2(p)$, while $n = p^s$.

For $s\geq3$ and $p\geq 5$, we have $s-1 \leq \log_2(s+1)$. Hence, $(s-1)\log_2(p) + \log_2(\frac{2}{3}) \geq 2\log_2(\frac{5}{4}) + \log_2(\frac{2}{3}) + (s-1).2 \geq 0 + 2.\log_2(s+1) > \log_2(s^2+s)$. As a result, $(s^2+s)\log_2p \leq (s^2+s)p \leq \frac{2}{3}p^s \leq p^s - 4$.

For $s \geq 3$ and $p = 3$, $n = 3^s \geq (s^2+s)\log_23 + 4 \geq m(H) + 4$

For $s \geq 6$ and $p = 2$, $n = 2^s \geq s^2 + s + 4 \geq m(H) + 4$

For $s = 2$ and $p \geq 5$, $p^2 \geq 6\log_2p + 4 \geq m(H) + 4$.

For $p = 3, s = 2$, $m(H) \leq 9 \leq n -4$.

For $p = 2, s = 3$, $|AGL_3(2)| = 2^6.3.7$ so $m(H) \leq 8 < n - 4$.

For $p = 2, s = 5$, $|AGL_5(2)| = 2^15.3^2.5.7.31$, so $m(H) \leq 20 < 2^5 -4 = n - 4$.

For $p = 2, s = 4$, $AGL_4(2) \cong C_{2^4}.A_8$, and by Whiston's lemma and the second isomorphism theorem for normal subgroup $N = C_{2^4}$ of $AGL_4(2)$, $m(H) \leq i(H \cap N) + m(H/H \cap N)  \leq i(N) + m(HN/N) \leq i(N) + i(AGL_4(2)/N) = i(N) + i(A_8) = 1 + 6 = 7 < n -4$ 
\vspace{5mm}

$\mathbf{Case\ 2}$: $H \leq T^k \cdot (Out(T) \times S_k )$ with $n = |T|^{k - 1}$. $T$ is non-abelian simple group so $|T| \geq 60$. Because $|T^k \cdot (Out(T) \times S_k )| \leq |T|^k \cdot |S_k||Out(T)| \leq |T|^k \cdot k^k|T|^{\log_2|T|}$, $m(H) \leq \lambda(H) \leq k\log_2|T| + k\log_2k + (\log_2|T|)^2$.

For $k \leq 3$, we have: $|T|^{k-1} = |T|^{k-2}|T| \geq k \cdot 4\log_2|T|, |T|^{k-1} \geq 4k^2 \geq 4k\log_2k$, and $|T|^{k-1} \geq |T|^2 \geq (2\log_2|T|)^2$

Hence $n - 4 = |T|^{k-1} - 4 \geq \frac{3}{4}|T|^{k-1} \geq  k\log_2|T| + k\log_2k + (\log_2|T|)^2 \geq m(H)$.
\vspace{5mm}

Let $d(G)$ be the minimal degree of a permutation representation of $G$. Now we need to handle the case where $H \leq Aut(T)$ where $T$ is simple and $d(T) \leq n$. By the classification of finite simple groups we have the following cases: 
\vspace{5mm}

$\mathbf{Case\ 3}$: $H \leq Aut(T)$, with $T$ = $PSL(m, q)$. We have $n \geq d(T) = \dfrac{q^m-1}{q-1}$, and $m(H) \leq \lambda(H) \leq \lambda(Aut(T)) \leq m^2\log_2(q) + 1$.

\vspace{5mm}
$\mathbf{Case\ 3a}$: For $m \geq 4$, $d(T) - 4 \geq q^{m-1} + 1$. We have $(m-2)\log_2q \leq 2\log_2m$ except for $(m, q) = (7, 2), (6, 2), (5, 2), (4, 3), (4, 2)$, and therefore, $(m-1)\log_2q \geq 2\log_2m + \log_2q$, and so $q^{m-1} \geq m^2q \geq m^2\log_2(q)$. As a result, $n - 4 \geq d(T) - 4 \geq q^{m-1} + 1 \geq m^2log_2(q) + 1 \geq m(H)$.

For $(m, q)  = (7, 2), (6, 2), (5, 2)$, $n - 4= \dfrac{q^m-1}{q-1} - 4 \geq m^2\log_2q + 1 \geq m(H)$. 

For $(m, q) = (4, 2)$, $|Aut(PSL(4, 2))| = 2^7.3^2.5.7$ so $m(H) \leq 11 = \dfrac{2^4-1}{2-1} - 4 \leq n -4$.

For $(m, q) = (4, 3)$, $|Aut(PSL(4, 3))| = 2^9.3^6.5.13$ so $m(H) \leq 17 < \dfrac{3^4-1}{3-1} - 4 \leq n -4$.

\vspace{3mm}
$\mathbf{Case\ 3b}$: $m = 3$. We have $\dfrac{q^3-1}{q-1} \geq 9\log_2q + 5\ \forall q \geq 5$, and so $m(H) \leq n - 4$. 

For $q = 2$, $|Aut(PSL(3, 2))| = 2^4.3.7$, so $m(H) \leq 6 < n -4$.

For $q = 3$, $|Aut(PSL(3, 3))| = 2^5.3^3.13$, so $m(H) \leq 9 \leq n - 4$.

\vspace{3mm}
$\mathbf{Case\ 3c}$: $m = 2$. We have $q \geq 4\log_2q + 4\ \forall q \geq 22$, and so $m(H) \leq n - 4$.

For $q \leq 5$, $m(H) \leq \floor{4\log_25} = 9 \leq n - 4$.

Now for other $q$, we show that $m(H) \leq 9 \leq n -4$ by considering the following tables: 
\vspace{3mm}
\begin{center}
	\begin{tabular}{| c | c | c | c | c |} 
		\hline
		$q$ & $7$ & $8$ & $9$ & $11$ \\
		\hline
		$|Aut(PSL(2, q))|$  & $2^4.3.7$ & $2^3.3^3.7$ & $2^5.3^2.5$ & $2^3.3.5.11$ \\
		\hline
		$m(H) \leq \lambda(H) \leq$ & 6 & 7 & 8 & 6 \\
		\hline
	\end{tabular}
\end{center}
\vspace{2mm}
\begin{center}
	\begin{tabular}{ | c | c | c | c | c |} 
		\hline
		$q$ & $13$ & $16$ & $17$ & $19$ \\
		\hline
		$|Aut(PSL(2, q))|$ & $2^3.3.7.13$ & $2^6.3.5.17$ & $2 ^5.3^2.17$ & $2^4.3^4.5$\\
		\hline
		$m(H) \leq \lambda(H) \leq$ & 5 & 9 & 8 & 9 \\
		\hline
	\end{tabular}
\end{center}

\vspace{5mm}
$\mathbf{Case\ 4}$: $H \leq Aut(T)$, with $T$ = $PSp(2m, q), m \geq 2$. $n \leq d(T) = \dfrac{q^{2m}-1}{q-1}$, and $m(H) \leq \lambda(H) \leq \lambda(Aut(T)) \leq (2m^2 + m + 1)\log_2(q)$.

We have $\dfrac{q^{2m} - 1}{q- 1} > q^{2m-1} + 4$. Now for $(m, q) \neq (3, 2), (2, 2), (2, 3)$, $(2m - 1)\log_2(q) \geq \log_2(2m^2 + m + 1)$, and so $(2m - 1)\log_2(q) \geq \log_2(2m^2 + m + 1) + \log_2(\log_2(q))$.
Therefore, $q^{2m - 1} \geq (2m^2 + m + 1)\log_2(q)$, and so $n - 4 \geq m(H)$. 

For $(m, q) = (3, 2), (2, 2)$ or $(2, 3)$, we have the following table:
\begin{center}
	\begin{tabular}{ |c | c | c | c | } 
		\hline
		T & $PsP(4,2)$ & $PsP(4, 3)$ & $PsP(6, 2)$ \\
		\hline
		$|Aut(T)|$ & $2^5.3^2.5$ & $2^7.3^4.5$ & $2^9.3^4.5.7$\\
		\hline
		$m(H) \leq \lambda(H) \leq$ & $8 < \frac{2^4 - 1}{2 -1} - 4$ & $12 < \frac{3^4 - 1}{3-1} - 4$ & $15 < \frac{2^6 - 1}{2-1} - 4$ \\
		\hline
	\end{tabular}
\end{center}

\vspace{5mm}
$\mathbf{Case\ 5}$: Now we consider the case when $T = P\Omega(n, q)$.

\vspace{3mm}
$\mathbf{Case\ 5a}$: $H \leq Aut(T)$, with $T$ = $P\Omega(2l+1, q), l \geq 3$, $q$ odd. 

For $q \geq 5$, $n \geq d(T) = \dfrac{q^{2l}-1}{q-1}$, and $m(H) \leq \lambda(H) \leq \lambda(Aut(T)) \leq (2l^2 + l)\log_2(q)$. Therefore, $d(T) > q^{2l-1} + 4$. $q^{2l-2} \geq 5^{2l-2} \geq 2l^2 + 2$, and so $n \geq d(T) \geq m(H) - 4$.

For $q = 3$, $n \geq d(T) = \frac{1}{2}3^l(3^l - 1) \geq 2(l^2 + l)\log_23 + 4 \geq m(H) -4$.

\vspace{3mm}
$\mathbf{Case\ 5b}$ : $H \leq Aut(T)$, with $T$ = $P\Omega^{+}(2l, q), l \geq 4$

For $q \neq 2$, $n \geq d(T) = \dfrac{(q^l-1)(q^{l-1} + 1)}{q-1}$, and $m(H) \leq \lambda(H) \leq \lambda(Aut(T)) \leq (2l^2 + 1)\log_2(q) + \log_26$.

We have $d(T) \geq (q^{l-1} + q^2 + q + 1)(q^{l-1} + 1) \geq (q^{l-1} + 8)(q^{l-1} + 1) > q^{2l-2} + 8$. Also note that $q \geq \log_2q$ and $q^{2l-3} \geq 3^{2l-3} \geq 2l^2 + 1$. Therefore $n \geq d(T) \geq q^{2l-2} + 8 \geq (2l^2 +1)\log_2q + 8 > m(H) + 4$. 

For $q = 2$, $n \geq d(T) = 2^{l-1}(2^l - 1) \geq 2l^2 + 7 > m(H) + 4$

\vspace{3mm}
$\mathbf{Case\ 5c}$: $H \leq Aut(T)$, with $T$ = $P\Omega^{-}(2l, q)$ with $l \geq 4, q \geq 3$ or $l \geq 5, q = 2$.

We have $n \geq d(T) \geq q^{2l-3}$, and $m(H) \leq \lambda(H) \leq \lambda(Aut(T)) \leq (2l^2 + 2)\log_2(q) + 1$.

Note that for $q \geq 3, l \geq 4$, $q^{2l-4} \geq 3^{2l-4} \geq 2l^2 + 4$, $log_2(q) \leq q$, and $5 \leq 2q$. Therefore, $m(H) + 4 \leq (2l^2+2)q + 2q \leq q^{2l-3} \leq n$.

For $q = 2, l \geq 4$, $n \geq 2^{2l-3} \geq 2l^2 + 7 \geq m(H) + 4$.

\vspace{5mm}
$\mathbf{Case\ 6}$: $H \leq Aut(T)$, with $T = U(n, q)$, and $m(H) \leq \lambda(H) \leq \lambda(Aut(T)) \leq  n^2\log_2(q+1) + 1$.

\vspace{3mm}
$\mathbf{Case\ 6a}$: $n \geq 5$ and $(n, q)$ is neither $(6m, 2)$ or $(6m, 3)$. We have $n \geq d(T) \geq \dfrac{(q^n - 1)(q^{n-1} - 1)}{q^2}$, and so $d(T) \geq q^{2n-3} - q^{n-2} - q^{n - 3} + 1 \geq (q-1)q^{2n-4} + 4 + 1$. 

For $q \geq 3$, $\log_2(q+1) \leq q -1$, and $q^{2n-4} \geq 3^{2n-4} \geq n^2$ for $n \geq 5$. Hence, $n \geq d(T) \geq n^2(q) + 5 \geq n^2\log_2(q+1) + 5 \geq m(H) + 4$. 

For $q = 2$, $2^{2n-4} \geq n^2\log_23 \ \forall n \geq 5$, so again we get $n \geq m(H) + 4$.

\vspace{3mm}
$\mathbf{Case\ 6b}$: $T = U(6m, 3)$. Then $d(T) \geq \frac{1}{4}(3^n - 1)3^{n-1} \geq 2n^2 + 5$ with $n \geq 6$, so $n \geq m(H) + 4$.

\vspace{3mm}
$\mathbf{Case\ 6c}$: $T = U(6m, 2)$. Then $d(T) \geq \frac{1}{3}2^{n-1}(2^n - 1) \geq n^2\log_23 + 5$ with $n \geq 6$, so $n \geq m(H) + 4$.

\vspace{3mm}
$\mathbf{Case\ 6d}$: $T = U(4, q)$. Then $d(T) \geq (q+1)(q^3 + 1) \geq 16\log_2q + 5$ with $n \geq 6$, so $n \geq m(H) + 4$.

\vspace{3mm}
$\mathbf{Case\ 6e}$: $T = U(3, q)$. Then $d(T) - 1 \geq q^3 \geq 9\log_2q + 4$ with $n \geq 6$, so $n \geq m(H) + 4$.

The case where $H$ is a subgroup of primitive wreath product and the case when $H \leq Aut(T)$ with $T$ is either Lie-type group (twisted Chevalley group or the Tits group) or sporadic group are already handled by \hyperlink{Whis}{Whiston}
\end{proof}

\begin{rmk}
Using GAP, Sophie Le analyzes primitive subgroup of $S_n$ for $n = \overline{9, 12}$, and obtained the result that $m(H) \leq n - 4$. Therefore, \cref{lm3} holds for $n \geq 9$.
\end{rmk}

\begin{lem}\label{lm4}
Let $\mathbf{P} = \{X_0, X_1, \dotsc , X_m\}$ be a partition of $\Sigma = \{1, 2,\dotsc, n\}$ with $m \geq 1$. Then for every $h \in S_{\Sigma} = S_n$, there exists $\alpha$, $\beta_0$ and $\{\beta_i\}_{i=1}^p$ such that $h = \alpha.\beta_0.\prod\limits_{i=1}^p \beta_i$, and $\alpha \in S_{X_1} \times S_{X_2} \times\dotsc. \times S_{X_m}$, $\beta_0 \in S_{X_0}$, and each $\beta_i \not\in S_{X_t} \ \forall t \in \overline{0, m}$ is a single cycle such that for each $j \in \overline{1, m}$, elements in $X_j$ that are not fixed by $h$ are all in the some same cycle $\beta_i$. Moreover, $\beta_0$ commutes with $\alpha$ and $\beta_i\ \forall i \in \overline{1, p}$, and $\beta_i$ with $i \in \overline{1, p}$ commutes with each other. We call this representation $h = \alpha.\beta_0.\prod\limits_{i=1}^p \beta_i$ the $M$-decomposition of $h \in S_n$. Also note that if we let $Y_i$ be the set of all $j \in \overline{1, m}$ such that elements in $X_j$ are in the cycle $\beta_i$, then $\mathbf{Q} = \{Y_1, \dotsc, Y_p\}$ is a partition of $\mu$ for some subset $\mu \subset \{1, 2, \dotsc, m\}$. We call this partition $\mathbf{Q}$ the associated partition of the $M$-decomposition of $h \in S_n$.

Now suppose that $\mathbf{P}$ is a partition $\{X, Y\}$ of $\Sigma$ instead. Then each element $h \in S_{\Sigma} = S_n$ can be represented as $\alpha.\beta$, where $\alpha \in S_X \times S_Y$ and $\beta$ is $1$ or is a single cycle of the form $(x_1, y_1, x_2, y_2, \dotsc, x_k, y_k)$ with $x_i \in X$, and $y_i \in Y$. We will call this representation the strong $M$-decomposition of $h \in S_n$. 
\end{lem}
\begin{proof}
Consider an element $h \in S_n$. Write the cycle decomposition of $h$. First, group all cycles that are in $\prod_{i=1}^m S_{X_i}$ into $\alpha_0$, and all cycles in $S_{X_0}$ into $\beta_0$. 

We say that $2$ cycles in are equivalent if they both have elements that are in some same $X_j$. Then we will have $p$ equivalent classes of cycles. Note that from the definition of equivalent relation, if we let $Y_i$ be the set of all $j$ such that some element in $X_j$ are in some cycle of the $i$th equivalent classes, then $Y_i$ are all disjoint and $\mathbf{Q} = \{Y_1,\dotsc,Y_p\}$, is therefore a parition of $\{1, \dotsc, m\}$. Then using the identity $(a, x_1 \dotsc, x_k)(b, y_1, \dotsc, y_l) = (a, b)(a, x_1, \dotsc, x_k, b, y_1, \dotsc, y_l)$ with $a, b \in X_j$, we can represent product of all cycles in the same $i$th equivalent classes as $\alpha_i \beta_i$, where $\alpha_i \in \prod_{j \in Y_i} X_j$, and $\beta_i$ is a single cycle containing elements in $X_0$ and $X_j$ with $j \in Y_i$. 

Therefore, we get $h = \alpha_0.\beta_0.\prod_{i=1}^p \alpha_i\beta_i$. Now note that $\alpha_j$ commutes with $\beta_k$ for $k \neq j$ so we can push all $\alpha_j$ to the top positions, and so $h = \alpha_0\alpha_1 \cdots \alpha_p\beta_0\beta_1 \cdots \beta_p$, and if we let $\alpha = \prod_{i=0}^{p}\alpha_i$, then we have the $M$-decomposition of $h$.

For the case $\mathbf{P} = \{X, Y\}$, based on $M$- decompostion of $h$, we will get $h = M. N$, where $M \in S_X \times S_Y$ and $N$ is $1$ or is a single cycle containing both elements of $X$ and $Y$.

Now suppose that $N \neq 1$ and has at least $2$ adjacent elements that are both $\in X$ or both $\in Y$. WLOG, assume that $N = (x_1, x_2, t_1, t_2, \dotsc, t_k)$ with $x_1, x_2 \in X$, then $N$ can also be rewritten as $(x_1, x_2)(x_2, t_1, t_2, \dotsc, t_k)$. Continuing this process, we can represent $N$ in term of $N_1.T$ where $T = (x_1, y_1, x_2, y_2, \dotsc, x_k, y_k)$ or $1$ with some $N_1 \in S_X \times S_Y$, and $x_i \in X$, $y_i \in Y$.

If we let $\alpha = M.N_1 \in S_X \times S_Y$, and $\beta = T$, we will get the strong $M$-decomposition of $h$.
\end{proof}

\begin{defi}\label{def3}
Given $m \in \mathbb{Z}$, and a partition $\mathbf{P} = \{X_0, X_1,\dotsc X_p\}$ of $\Sigma = \{1, 2,\dotsc, m\}$, we say that a subgroup $G$ of $S_m = S_{\Sigma}$ has $M$-property wrt $\mathbf{P}$ if for any $h \in G$, and $\sigma \in \prod\limits_{i=1}^p S_{X_i}$, $\sigma h \sigma^{-1} \in G$

Now for a subset $T$ of $S_m$, we define $G(T, \mathbf{P})$ to be the smallest subgroup of $S_m$ that contains $T$ and has $M$-property wrt $\mathbf{P}$. It is easy to see that $G(T, \mathbf{P})$ is the subgroup generated by $\{\sigma h \sigma^{-1}, h \in T, \sigma \in\prod\limits_{i=1}^p S_{X_i}\}$
\end{defi}

\begin{lem}\label{lm5}
Suppose $m \in \mathbb{Z}$, and $S_m = S_{\Sigma'}$ with $\Sigma \subset \Sigma' = \{1, 2,\dotsc, m\}$. Let $\mathbf{P} = \{\Sigma'\setminus\Sigma, X, Y\}$ be a partition of $\Sigma'$ such that $|\Sigma| > 4$. Assume that a subgroup $G$ of $S_m$ has $M$-property wrt $\mathbf{P}$, and $\alpha(aby) \in G$ for some $\alpha \in S_X \times S_Y$, $a, b \in X$, and $y \in Y$. Then $A_{\Sigma} = A_{X \cup Y} \subset G$
\end{lem}

\begin{proof}
Starting with arbitrary $\alpha \in S_X \times S_Y$, we will try to find $\alpha'$ with much simpler form such that $\alpha'(a, b, y)$ is still in $G$. 

First, we notice that $(a, b, y).\alpha \in G$. Hence $\alpha(a, b, y)(a, b, y)\alpha \in G$, and therefore $\alpha^2(a, y, b) \in G$. Repeating this argument, we get $\alpha^{4}(a, b, y) \in G$. So first, we can always eliminate all transpositions, if any, in the cycle decomposition of $\alpha$. 

Moreover, if some cycle $(\dotsc, m, n, \dotsc)$, which is not a transposition, appears the cycle decomposition of $\alpha$ such that $m, n$ is not $a,b$, and all elements in this cycle are in $X$, then we can consider $\sigma \in S_X$ that swap $m$ and $n$. 

We have 
\begin{equation*}
\begin{split}
\alpha(\sigma \alpha\sigma^{-1})^{-1} & = (x_1, \dotsc, x_k, m, n, y_1, \dotsc, y_l).(y_l, \dotsc, y_1, m, n, x_k, \dotsc, x_1) \\
& = (n, m, y_1) \in G. 
\end{split}
\end{equation*}
Therefore, $A_X \subset G$. Using this observation, we can eliminate all even permutation of $S_X$, and therefore all elements of $S_X$ that are in the cycle decomposition of $\alpha$. 

As a result, we can assume that in the cycle decomposition of $\alpha$, its only possible non-identity term in $S_X$ is a single $3$-cycles that contains $a, b$ and some other element $c \in X$. Using similar argument, we can further assume that $\alpha$ contains no non-identity element of $S_Y$. Therefore, we must have either $(a, b, y) \in G$ or $(a, b, c)(a, b, y) = (a, c)(b, y) \in G$ or $(a, c, b)(a, b, y) = (b, y, c) \in G$.

In short, we have $(a, b)(c, y) \in G$ or $(a, b, y)\in G$, with some $a, b, c \in X, y \in Y$. 

If $(a, b, y) = (y, a, b) \in G$, then $(y, c, d) \in G$ with any $c, d \in X$, and so $A_{X \cup \{y\}} \in G$. As a result, $(a, b, c) \in G$ with any $c \neq a, b$ in $X$, if any. Since $(a, b, y) \in G$, $(a, b, z)  \in G$ for each $z \in Y$. Therefore, $(a, b, t) \in G$ with every $t \neq a, b$. Hence $A_{\Sigma} \subset G$. 

The case when $(a, b)(c, y) \in U$ is handled in a similar fashion.
\end{proof}
\begin{lem}\label{lm6}
Let $H$ be a subset of $S_n$ ($n \geq 10$), and $\mathbf{P} = \{X, Y\}$ is a partition of $\Sigma = \{1, 2,\dotsc, n\}$. There exists a subset $K$ of $H$ with $|K| \leq 9$ such that $H \subset G(K, \mathbf{P})$
\end{lem}
\begin{proof}
$\mathbf{P}$ is fixed throughout the proof of this lemma so we can just write $G(T)$ instead of $G(T, \mathbf{P})$.

Now we consider the following cases:

$\mathbf{Case\ 1}$: Now assume that for some $h_{i_0}$, the $\beta$ component in its $M$-decomposition (\cref{lm4}) is not $1$ and doesn't contain some $b \in X$. We have $\beta = (\dotsc, a, c, \dotsc)$ for some $a\in X, c\in Y$. Consider element $\sigma \in S_X \times S_Y$ that swaps $a$ and $b$. 

Let $U = G(\{h_{i_0}\})$. Then $\alpha_2\sigma\beta\sigma^{-1} = \sigma h_{i_0} \sigma^{-1} \in U$ with $\alpha_2 = \sigma\alpha\sigma^{-1} \in S_X \times S_Y$. Therefore, $\alpha \beta (\sigma\beta\sigma^{-1})^{-1}(\alpha_2)^{-1} \in U$. Hence, $\alpha_2^{-1}\alpha\beta(\sigma\beta\sigma^{-1})^{-1} \in U$. 

In other word, 
\begin{equation*}
\begin{split}
\alpha_3\beta(\sigma\beta\sigma^{-1})^{-1} &= \alpha_3(x_1, \dotsc, x_k, a, c, y_1, \dotsc, y_l)(y_l, \dotsc, y_1, c, b, x_k, \dotsc, x_1) \\
& =  \alpha_3(cba) \in U
\end{split}
\end{equation*} for some $\alpha_3 = \alpha_2^{-1}\alpha \in S_X \times S_Y$. Hence, its inverse $(abc).\alpha_3^{-1}  \in U$. 

As a result, $\alpha_3^{-1}(a, b, c) \in U$, and by \cref{lm5}, $A_n \subset U$.

Now if all $h_j$ are even permutation then $U$ already contains all $h_j$. Otherwise, we just need at most one element $h_{j_0}$ together with $U$ to generate the whole $S_{\Sigma}$, and therefore generate all $h_i$. Therefore, our set $K$ needs at most 2 elements $h_{i_0}$ and $h_{j_0}$.

$\mathbf{Case\ 2}$: Now we only need to handle the case when the $M$-decomposition of $h \in H$ has the form $h = \alpha\beta$ with $\alpha \in S_X \times S_Y$, and $\beta = (x_1, y_1, \dotsc, x_k, y_k)$, where $k = |X| = |Y|, x_i \in X, y_i \in Y$ or $\beta = 1$. Note that $n \geq 10$ so $k \geq 5$.

Let $H_1 = \{h_i \in H: \beta = 1\}$, and $H_2 = \{h_i \in H: \beta = (x_1, y_1, \dotsc, x_k, y_k)\}$. We now prove that $K$ needs at most four $h_i \in H_1$, and needs five more $h_j \in H_2$ to generate all elements in $H = H_1 \cup H_2$. 

Suppose there exists some $h_1 = \alpha = A.B \in S_X \times S_Y$, where $A, B$ are non-trivial permutations of $S_X$ and $S_Y$. Then consider $U = G(\{h_1\})$. For any $\sigma \in S_X$, $\sigma h_1\sigma^{-1}h_1^{-1} = \sigma A \sigma^{-1}A^{-1} \in U$. Since $A$ is not identity, there exists $\sigma$ such that $\sigma A \sigma^{-1}A^{-1}$ is a non-trivial permutation of $S_X$. Therefore, $A_X \in U$. Similarly, $A_Y \in U$. We say that $h_i \sim h_j$ with $h_i, h_j \in H_1$ if $h_ih_j^{-1} \in A_X \times A_Y$. Note that, for any equivalent class $C$, $C \subset G(\{g_i\}\cup\{h_1\})\ \forall g_i \in C$. Therefore, we need $\leq 3$ other $h_i$ to generate the whole $H_1$ because there are at most $4$ equivalent classes, and so $K$ only needs $\leq 4$ elements of $H_1$ to generate all elements in $H_1$. If $H_1 \in S_X \cup S_Y$, then by similar arguments, $K$ only needs $\leq 2 + 2 = 4$ elements from $H_1$.

Now we will handle $H_2$. Choose a fixed cycle $C = (u_1, v_1, \dotsc, u_k, v_k), u_i \in S_X, v_i \in S_Y$. For each $h_i = \alpha(x_1, y_1, \dotsc, x_k, y_k)$, there exist $\sigma_i \in S_X \times S_Y$ such that: $\sigma_ih_i\sigma_i^{-1} = (\sigma_i\alpha\sigma_i^{-1})(\sigma_i(x_1, y_1, \dotsc, x_k, y_k)\sigma_i^{-1}) = \alpha_i.C$. If $\alpha_i$ are the same for all $i$, then $G(\{h_i\}) $ contains $H_1$ for any $h_i \in H_1$.

Now suppose there is some $i_0$ and $j_0$ such that $ \alpha_{i_0} \alpha_{j_0}^{-1} = A.B$, $A$ and $B$ are non-trivial permutations in $S_X$ and $S_Y$. Then again, $A_X$, and $A_Y$ are both subsets of $G(\{\alpha_{i_0}, \alpha_{j_0}\})$. 

We say that $h_m \sim h_n$ if  $\alpha_m\alpha_n^{-1} \in A_X \times A_Y$. Using the same argument for $H_1$, we only need at most 3 more $h_k$ so that $G(\{h_{i_0}, h_{j_0}, h_{k_1}, h_{k_2}, h_{k_3}\})$ contains $H_2$. 

The case $\alpha_{i} \alpha_{j}^{-1} \in A_X \cup A_Y \ \forall i, j$ is handled in a similar way.

As a result, our set $K$ needs $\leq 4 + 5 = 9$ elements. 
\end{proof}
\begin{thm} \label{thm2}
Suppose that $H \neq A_n$ is a proper subgroup of the symmetric group $S_n$, and $H$ is transitive. Then $m(H) \leq n - 4 \ \forall n \geq 25$. If $H$ is a proper subgroup of $A_n$, then $m(H) \leq n-4 \ \forall n \geq 9$.
\end{thm}
\begin{proof}
First, by \cref{lm3}, we only need to handle the case when $H$ is an imprimitive but transitive subgroup of $S_n$. Therefore, we can assume that $H \leq S_{\Gamma} \wr S_{\Delta}$. Using the same notation as in \cref{def2}, we have $n$ copies of $\Gamma$ or $n$ blocks, $\Gamma_1$, $\Gamma_2,\dotsc, \Gamma_{|\Delta|}$, and an irredundant generating set of length $m(H) = l$, $\{g_1,\dotsc, g_l\}$, so that $\{g_1, \dotsc, g_k\}$ generates block permutations while $K = \{g_{k+1}, \dotsc, g_l\}$ contains elements fixing the blocks. 
	
The proof of \cref{thm2} is based on Whiston's proof for $S_n$ with one exceptional case that is handled by \cref{lm6}.

Whiston proved that for $n \geq 7$, if $H \neq A_n$ is a transitive subgroup of the symmetric group $S_n$, and is not a subgroup of $S_2 \wr S_{\frac{n}{2}}$, then $m(H) \leq n - 3$. Cameron and Cara further showed that $m(H) \leq n - 3$ even if $H \leq S_2 \wr S_{\frac{n}{2}}$. Using GAP, Sophie Le checked for $S_5$ and $S_6$ and found that no transitive subgroup of $S_n$ other than $A_n$ has $m(H) > n -3$ for $n = \overline{5, 6}$. Therefore, if $H \neq A_n$ is a transitive subgroup of $S_n$, $n \geq 5$, then $m(H) \leq n - 3$.

Let $X_1$ be the smallest subset of $\{g_{k+1},\dotsc,g_l\}$ such that $X_1$ generates the action of $K$ on $\Gamma_1$, and suppose that $X_1 = \{g_{k+1},\dotsc,g_r\}$. Then there exists $h_i$ generated by $X_1$ such that $g_i' = h_ig_i$ with $i > r$ fixes $\Gamma_1$ pointwise. Let $K_1 = \{g_i', i >r\}$. We have $K_1$ is irredundant because $K$ is irredundant. Now take $X_2$ to be a subset of $K_1$ that generates the action on $\Gamma_2$. Then we will have an irredundant set $K_2$ that fixes $\Gamma_1$ and $\Gamma_2$. We continue this process to build $X_i$ that generates the action of $K_{i-1}$ on $\Gamma_i$, and $K_i$ be the set built from remaining $g_i$: $K_i = \{t_jg_j, t_j \in \genby{X_1 \cup \dotsc \cup X_i}$. Notice that $K_i$ must be irredundant, and therefore $K_{|\Delta|}$ must be an empty set, and so we have $l-k = |K| = \sum_i|X_i|$. Hence, $l - k \leq \sum_i m(X_i)$. We now consider several cases to prove the inequality $m(H) \leq n - 4$:

$\mathbf{Case\ 1}$: If there is no $i$ such that the subgroup generated by $H_1= \{g_{k+1},\dotsc, g_l\}$ acts as $S_{\Gamma}$ or $A_{\Gamma}$ on $\Gamma_i$. 
Since $X_i$ cannot acts as $S_{\Gamma}$ on $\Gamma_i$, $m(X_i) \leq |\Gamma|-2$, so we get $l - k \leq \sum_i m(X_i) \leq  |\Delta|(|\Gamma| - 2)$, and therefore $m(H) = l \leq n - |\Delta| -1 \leq n - 4$ for $|\Delta| \geq 3$.

Now suppose that $|\Delta| = 2$. Then $|\Gamma| \geq 5$. If either $X_1$ or $X_2$ acts transitively on $\Gamma_1$ or $\Gamma_2$, then $m(X_1)$ or $m(X_2) \leq |\Gamma| - 3$. Hence $l \leq 1 + |\Gamma| - 2 + |\Gamma| - 3 = n -4$. So now we just need to consider the case where $X_1$ and $X_2$ both acts intransitively on $\Gamma_1$ and $\Gamma_2$. In this case, there are partitions $\mathbf{P_i} = \{U_i, V_i\}$ of $\Gamma_i$ such that $X_i$ acts seperately on the sets $U_i, V_i$. Therefore $m(X_i) \leq m(X_i^{U_i}) + i(X_i^{V_i}) \leq |U_i| -1 + |V_i| - 1 = |\Gamma| - 2$ ($X_i^T$ is the group action of $X_i$ on $T$). Hence, $l \leq 1 + |\Gamma| - 2 + |\Gamma| - 2 = n - 3$. 

If the equality occurs, then $X_i$ must acts as $S_{U_i} \times S_{V_i}$ on $\Gamma_i$ for $i = \overline{1, 2}$. If $H$ is a subgroup of $A_n$, this case cannot happen.

If, however, $H$ is just a subgroup of $S_n$, then we will use \cref{lm6} to handle this case. 

Let $h_{21}$ be the element that move block $2$ to block $1$. Suppose that $h_{21}$, after moving the block, performs a permutation $\sigma$ on $\Gamma_1$. Let $h_{\delta}$ be an element in $L$, the subgroup generated by $\{ h_{21}^{-1}g_{k+1}h_{21}, \dotsc, h_{21}^{-1}g_rh_{21}\}$ such that $h_{\delta}$  fixes the blocks and performs permutation $\delta$ on $\Gamma_2$. Note that $\delta \in \sigma^{-1}X_1^{\Gamma_1} \sigma = S_{U_1'} \times S_{V_1'}$ for some partition $\{U_1', V_1'\}$ of $\Sigma$. 

Now the elements $h_{\delta}^{-1}g_i'h_{\delta}$ (here $g_i' = h_ig_i \in X_2 = K_1$) will again fixes $\Gamma_1$ and performs the permutation $\delta^{-1}g_i'\delta$ on $\Gamma_2$. Then by \cref{lm6}, we only need $\leq 9$ elements of $X_2$ together with $\{g_1,\dotsc,g_r\}$ to generate the whole $X_2$. Hence we must have: $m(H) \leq 1 + n/2 - 1 + 9 \leq n - 4$ for $n \geq 25$. 

$\mathbf{Case\ 2}$: If $|\Gamma| > 2$, and there is some $i$ such that $H$ acts as $A_n$ or $S_n$ on $\Gamma_i$. Then by \cref{lm2}, $m(H) \leq |\Gamma| + 2 |\Delta| - 3 = n - (|\Gamma| - 2)(|\Delta| - 1) - 1 \leq n - 3 - 1 = n - 4$ for $n = |\Gamma||\Delta| \geq 10$ and $|\Gamma| > 2$. 
 
For $n = 9$, $|\Gamma| = |\Delta| = 3$. WLOG, suppose $i = 1$. If $m(H) = n - 3 = 6$, then the block permutation of $H$ must be isomorphic to $S_3$, which is $2$-transitive, and then by \cref{lm21}, we only need at most $5$ elements to generate the whole $H$. Contradiction. Hence $m(H) \leq 5 \leq n -4$, and so the case $n = 9$ is also handled.

$\mathbf{Case\ 3}$: Now we need to handle the case when $|\Gamma| = 2, |\Delta| \geq 5$.
By \cref{lm15} (for $k = 4$), and the fact that $m(K) \leq n - 3$ for all transitive subgroup $K \neq A_n \leq S_n$ with $n \geq 5$, $m(H) \leq 2|\Gamma| - 4 = n - 4$.
\end{proof}

\begin{cor}\label{cor1}
$A_n$ is strongly flat for $n \geq 9$.	
\end{cor}
\begin{proof}
For any proper subgroup $H$ of $A_n$, we will prove that $m(H) < n - 2 = m(A_n)$.

By \cref{thm2}, we only need to consider the case when $H$ is intransitive. Then there exists a partition $\mathbf{P} = \{X, Y\}$ of $\Sigma = \{1, 2,\dotsc, n\}$ so that $H \subset S_X \times S_Y \leq S_n = S_{\Sigma}$. Suppose $\pi_X: S_X \times S_Y \to S_X$ is the projection from $S_X \times S_Y$ to its first factor, and $H^X = \pi_X(H)$. Similarly, we have $\pi_Y$, and $H^Y$. Therefore, $m(H) \leq m(H^X) + m(H^Y) \leq |X| - 1 + |Y| - 1 = n - 2$. The equality holds if and only if $H^X = S_X$ and $H^Y = S_Y$, and therefore, $H$ must be $S_X \times S_Y$. Contradiction since $S_X \times S_Y$ is not a subgroup of $A_n$. Therefore, $m(H) \leq n - 3$, and we finish the proof for \cref{cor1}.
\end{proof}

Now to classify all irredundant generating set of size $n-2$ in $A_n$, we will consider graphs which are similar to the one used in Cameron and Cara's paper.
\begin{defi}
Let $G$ be the group $S_n$ or $A_n$. Suppose $\{h_1, h_2,\dotsc,h_{n-2}\}$ is an irredundant generating set of length $n - 2$ of $G$. Let $G_i$ be the proper subgroup of $S_n$ generated by $(h_i)_{i \neq j}$. Then by \cref{thm2}, for $G = A_n$, and $n \geq 9$, or for $G = S_n$ and $n \geq 25$, $G_i$ is an intransitive subgroup if $G_i \neq A_n$. Now we will build the $\mathbf{original\ graph}\ Gr_o$. We start with the set of vertices $\Sigma = \{1, 2, \cdots, n\}$. 

Now for each $i \in \overline{1, n-2}$ such that $G_i$ is intransitive, there exists $x$ and $y$ in $\Sigma$ such that $h_i$ maps $x$ to $y$, and $x, y$ are not in the same $G_i-orbit$. Such $x, y$ exist because if $x, y$ are in the same $G_i-orbit$ whenever $h_i$ maps $x$ to $y$, then the orbits of $G$ are the same as the orbits as $G_i$. This is impossible because $G_i$ is intransitve while $G$ is transitive. Add the edge that connects $x$ to $y$, $(x,y)$, to the graph and named this edge $h_i$. Then we have a graph $G_o$ with either $n - 2$ or $n - 3$ edges. Here note that the edge $(x, y)$ is different from the transposition $(x, y)$ even though they are related to each other.
\end{defi}

\begin{lem}
The graph $G_o$ is a forest
\end{lem}
\begin{proof}
If there is a cycle, then consider an edge $(x, y)$ that correponds to some $h_i$ in this cycle. Then we can go around the cycle from $x$ to $y$ through other edges correponding to $h_j$ with $j \neq i$. This leads to contradiction because $x$ and $y$ are not in the same $G_i$-orbit.
\end{proof}
\begin{lem}
If $G = A_n$, then the corresponding forest $Gr_o$ has $n-2$ edges. Moreover, each element $h_i$ is the product of exactly 2 transpositions, one of which correponds to an edge lying in $Gr_o$
If $G = S_n$ then there are two possible cases:
\begin{itemize}[leftmargin=0.1 in]
	\item There is exactly one $G_i = A_n$, and we can assume $i = 1$. Then $G_o$ has $n-3$ edges. Moreover, $h_1$ is a transpositions, and other $n -3$ elements $h_i$ are product of exactly $2$ transpositions, one of which already appear in $Gr_o$.  
	\item Or if for every $i$, $G_i \neq A_n$, then $Gr_o$ has $n-2$ edges, and each element $h_i$ is product of $\leq 2$ transpositions, one of which corresponds to an edge lying in $Gr_o$. 
\end{itemize}
Finally, let $E_1(h_i)$ be the edge that connects $x$ to $y$ and lies in $Gr_o$ and corresponds to one transposition $(x, y)$ of $h_i$, and let $E_2(h_i)$ be the edge that correponds to the other transposition, if any. For the transposition $h_1$ in $S_n$ case, we just use the notation $E_1(h_1)$
\end{lem}
\begin{proof}
Consider a particular $h_i$, and write the cycle decomposition of $h_i$. Then represent each cycle $(u_1, u_2, \dotsc, u_k)$ in the decompostion in term of products of transpositions $(u_1, u_2), (u_2, u_3),\dotsc,(u_{k-1}, u_k)$. Now, add to the original graph $Gr_o$ edges that connect these $u_i$ to $u_{i+1}$. Suppose that there's a cycle in the resulting graph. If this cycle contains some edges $\{x, y\}$ that corresponds to $h_j$ with $j \neq i$, then by going along the other edge of this cycle, we can connect $x$ to $y$ by edges corresponding to elements in $G_j$, and this is again impossible. Therefore, the only edges that can be in the cycle is the edges built from cycle decomposition of $h_i$ as we described above. But those edges clearly cannot create a cycle. Therefore, the resulting graph must be forest and can only have at most $n -1$ edges. 

Now if $G = A_n$ or $G = S_n$ and $G_i \neq A_n \forall i$, then $Gr_o$ already has $n - 2$ edges. So one can only add at most one edge from the cycle decomposition of some $h_i$. Hence, $h_i$ can be product of at most $2$ transpositions, and one of those transpositions has already appeared in the graph $G_o$. In the case $G = A_n$, $h_j$, therefore, must be product of exactly $2$ transpositions.

If $G = S_n$ and $G_1 \neq A_n$, then there are already $n-3$ edges corresponding to $h_i$ with $i \neq 1$ in $G_o$. If we add $k$ transpositions from $h_1$'s decomposition, then by a similar argument as above, the new graph is still a forest, and $k + n -3 \leq n -1$ or $k \leq 2$. Since $h_1$ is an odd permutation, $k$ must be $1$, and, therefore, $h_1$ is a transpositions. Similarly, for $i \neq 1$, $h_i$ must be product of $\leq 3$ transposition. Because $h_i$, $i \neq 1$, is an even permutation, it must be the product of exactly $2$ transpositions.
\end{proof}
\begin{thm}\label{thm3}
All irredundant generating sets of length $n - 2$ in $A_n$, $H = \{h_1, h_2,\dotsc, h_{n-2}\}$, has the form $\{sg_1^{e_1}, sg_2^{e_2},\dotsc,sg_{n-2}^{e_{n-2}}, e_i \in \{1, - 1\}\}$, where $s, g_i$ are transpositions such that the graph associated with these transpositions is a tree. 
\end{thm}
\begin{proof}
Now start with the forest (original graph) $Gr_o$ associated to $G = A_n$ that is defined above. If all $E_2(h_i)$ are the same and are some edge $E$, then we can add $E$ to $G_o$, and the resulting graph is still a forest, and, therefore, is a tree. Therefore the irredundant generating set $H = \{h_i\}$ has the form as desired. 

Now for some $j \neq i$ such that $E_2(h_i) \neq E_2(h_j)$, add $E_2(h_i)$ and $E_2(h_j)$ to the forest $Gr_o$. Then we will get $n$ edges in the resulting graph, and there must be some cycle $C(i, j)$ in this new graph. Notice that $E_1(h_k)$ for $k \neq i, j$ cannot be in this cycle, so the cycle has $\leq 4$ edges. Also, note that the cycle must has at least $3$ edges so it must contains both $E_1(h_t)$ and $E_2(h_t)$ for $t = i$ or $t = j$. WLOG, suppose that the cycle $C(i, j)$ contains $E_1(h_i)$ and $E_2(h_i)$, and has $4$ vertices $1, 2, 3, 4$. Now we have $2$ following cases :

$\mathbf{Case\ 1}$: $E_1(h_i)$ and $E_2(h_i)$ share no common vertex. Then the cycle must have at least $4$ edges and therefore have exactly $4$ edges $E_k(h_t), k = \overline{1, 2}, t \in \{i, j\}$. WLOG, suppose that $E_1(h_i) = (1, 2), E_2(h_i) = (3, 4), E_1(h_j) = (2, 3)$ and $E_2(h_j)= (1, 4)$. Then $3$ and $4$ must be in different $G_i$-orbits while $1$ and $4$ must be in different $G_j$-orbit . 

Clearly, $E_2(h_k) \neq E_2(h_i)$, for $k \neq j$, so we can consider the cycle $C(i, k)$ instead. $C(i, j)$ must contain $E_2(h_k)$. This is impossible because no two vertices in $\{1, 2, 3, 4\}$ can be in the same $G_i$ or $G_j$-orbits and cannot be connected by an edge correponding to $h_k$ with $k \neq i, j$.

\vspace{5mm}
{\centering
\includegraphics[scale = 0.45]{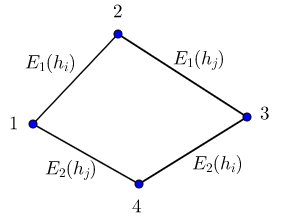}
\par}
\vspace{5mm}

$\mathbf{Case\ 2}$: If two edges of $h_i$ share a common vertex. Again, WLOG, suppose $E_1(h_i) = (1, 2)$, and $E_2(h_i) = (1, 3)$. Hence $1$ and $3$ are not in the same $G_i$-orbit. For any $k \neq i, j$, we again consider the cycle $C(i, k)$. $1$, $2$, and $3$ are in different $G_i$-orbits, so $C(i, k)$ must contain both $E_1(h_i)$ and $E_2(h_i)$. Therefore, this cycle has only two possible forms:

\begin{center}
\includegraphics[scale=0.45]{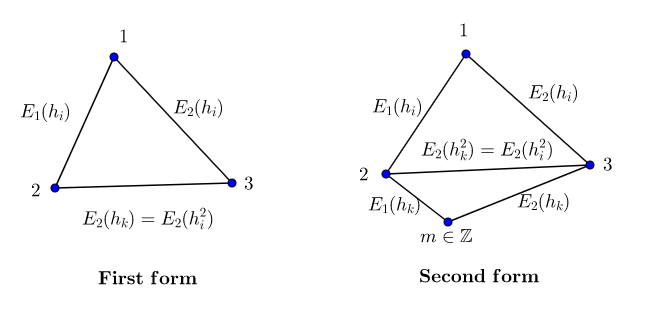}
\end{center}

First, we will change $h_i$ to $h_i' = h_i^2$ so that $E_2(h_i') = (2, 3)$. In the second form, if any, we can change $h_k$ to $h_k' = h_k^2$ so that all edges $E_2(h_t')$ with the new $h_t' = h_t$ or $h_t^{2}$ are the same and are the edge $(2, 3)$. As a result, $\{h_j'\}$ and $\{h_j\}$ has the form indicated.
\end{proof}

\begin{rmk}
Sophie Le found that \cref{thm3} is also true for $A_7$ and $A_8$. It is easy to see that \cref{thm3} is true for $A_4$ and $A_3$, and therefore \cref{thm3} is true for all $n \neq 5, 6 \geq 3$.
\end{rmk}

\begin{lem}\label{lm9}
Let $\{X, Y\}$ be a partition of $\Sigma = \{1,2,..,n\}$. Suppose that $G = (S_X \times S_Y) \cap A_n$. Then either $G$ is maximal subgroup of $A_n$ or $|X| = |Y| = \frac{n}{2}$ and there is an unique maximal subgroup $M = (S_{\frac{n}{2}} \wr S_2) \cap A_n$ that contains $G$.
\end{lem}
\begin{proof}
Suppose we have an element $h \in A_n \setminus G$. 

Write the $M$-decomposition of $h$: $h = M.N$, where $M \in S_X \times S_Y$, and $N = (x_1, y_1, \dotsc, x_k, y_k)$ with $x_i \in X, y_i \in Y$.

Let $F$ be the subgroup generated by $h$ and $G$. Suppose that there is some $x \in X$ that not in the cycle $N$ of $h$. Let $\sigma \in S_X$ be the permutation that swaps $x_1$ and $x$. Because $h \in F$, $h' = (ab). N \in F$, where $(ab)$ is some transpostion in $S_X$. We have $N.(a, b)(a, b).N = N^2 = (x_1, \dotsc, x_k)(y_1, \dotsc, y_k) \in G$, so $N.(a, b)$ is also in $F$. 

Because $h \in F$, and $\sigma^{-1}(a, b) \in G$, $M.\sigma^{-1} \sigma N(\sigma^{-1}(a, b)) \in F$. Hence $\sigma N\sigma^{-1}(a, b) = (x, y_1, \dotsc, x_k, y_k)(a, b) \in F$ since $M.\sigma^{-1}$ is in $G = (S_X \times S_Y) \cap A_n$. But we also have $N.(a, b) \in F$, so $(a, b).N^{-1} \in F$. 

Therefore, $(x, y_1, \dotsc, x_k, y_k)(a, b)(a, b)(y_k, x_k, \dotsc, y_1, x_1)  = (y_1, x_1, x) \in F$. As a result, $A_n = F$.

Consequently, if $G$ is not maximal, and $F_0$ is some maximal subgroup containing $G$, then $F_0$ can only contain elements of the form $M(x_1, y_1, \dotsc, x_k, y_k)$, where $M \in S_X \times S_Y$, and $\{x_i, y_i\}_{i = 1}^k$ is exactly the set $\Sigma$. Hence, $F_0$ must be $(S_{\frac{n}{2}} \wr S_2) \cap A_n$. Moreover, because of the same reason, $F_0 = (S_{\frac{n}{2}} \wr S_2) \cap A_n$ must also be a maximal subgroup of $A_n$.
\end{proof}

\begin{thm}
Let $\{h_1, h_2,\dotsc,h_{n-2}\}$ be any irredundant generating set in $A_n$ of length $n-2$ ($n \neq 5, 6 \geq 3$). There is an unique collection of maximal subgroups $M_1, M_2,\dotsc, M_{n-2}$ that are certified by $h_i$. In other word, $A_n$ satisfies the uniqueness property
\end{thm}
\begin{proof}
By \cref{thm3}, we can assume that $h_i = ss_i$, with $i = \overline{1, n-2}$ so that the edge $E(s) = E_2(h_i)$ and $E(s_i) = E_1(h_i)$ corresponding to transpositions $s$ and $s_i$ create a tree $T$ on the set of vertices $\Sigma = \{1,2,\dotsc,n\}$. 

Again, if $G_i$ be the subgroup generated by $\{h_j\}_{j \neq i}$, then we have $G_i = S_X \times S_Y \cap A_n$, where $X$ and $Y$ are sets of vertices of two connected components of $T \setminus \{E_1(h_i) = E(s_i)\}$. Then by \cref{lm9}, $G_i$ is almost maximal. Therefore, the uniqueness property is satisfied.
\end{proof}

\begin{defi}
A group $G$ with $m = m(G)$ satisfies the replacement property if for any irredundant generating set $H = \{h_1, h_2,\dotsc, h_m\}$ of length $m$ of $G$ and any non-identity element $h \in G$, there exists some $i \in \overline{1, m}$ so that the set $\{h_j\}_{j \neq i} \cup \{h\}$ still generates $G$.
\end{defi}
\begin{thm}\label{thm5}
$A_n$ satisfies the replacement property for all $n > 0$.
\end{thm}
\begin{proof}
Using GAP, G. Frieden and Sophie Le found that $A_n$ satisfies replacement property for $n = \overline{1, 8}$. So we just need to consider the case when $n \geq 9$.

Again, by \cref{thm3}, we can consider an irredundant generating set of length $n-2$ of the form $H = \{h_i = ss_i, \ i = \overline{1, n-2}\}$ so that edges $E(s)$ and $E(s_i)$ corresponding to transpositions $s$ and $s_i$ create a tree $T$ on the set of vertices $\Sigma = \{1,2,\dotsc,n\}$. 

Consider any element $h \in A_n$. Then there are at least $2$ pair of elements in $\Sigma$, $(x, y)$ and $(z, t)$, such that $h$ maps $x$ to $y$, and maps $z$ to $t$. Consider the paths $L_i$ in $T$ that connect $x_i$ to $y_i$ for all distinct pairs $(x_i, y_i)$ such that $h$ maps $x_i$ to $y_i$. 

Suppose that there exists two edges $E(s_i)$ and $E(s_j)$ with $s_i, s_j \neq s$ in $\cup_{i =1}^{k}L_i\ (k \geq 2)$. Then $h \not\in G_i$ and $h \not\in G_j$. By \cref{lm9}, either $G_i$ or $G_j$ must be maximal because $(S_{\frac{n}{2}} \wr S_2) \cap A_n$ cannot contain $G_i \cup G_j = A_n$. WLOG, suppose that $G_i$ is a maximal subgroup. Then we can replace $h_i$ by $h$.

Otherwise, if $\cup_{i =1}^{k} L_i= \{E(s_i), E(s)\}$ for some $i$, then $k = 2$, and $h$ must be $ss_i$, and we can replace $h_i$ by itself.
\end{proof}

\section{Properties of irredundant generating set of arbitrary length in the symmetric group}

\subsection{Classification of irredundant generating sets of length $n - 2$ of symmetric group}

Cameron and Cara classified all irredundant generating sets of length $n - 1$ in $S_n$. Now we will classify all irredundant generating sets of length $n - 2$ in $S_n$. 
\begin{defi}
Let $\Sigma = \{1,2,\dotsc,n\}$. For a subset $S \subset S_n$ consisting of transpositions, we define $Gr_S$ to be the graph on the set of vertices $\Sigma$ such that $x$ and $y$ is connected to by an edge $E(s)$ if there exists some transposition $s \in S$ such that $s = (xy)$. 
\end{defi}
\begin{thm}
(Classification of irredundant generating sets of length $n-2$ for $S_n$ with $n > 24$)

Let $H = \{h_1, h_2,\dotsc, h_{n-2}\}$ be an irredundant generating set of length $n-2$ of $S_n$. Then we can change some of the element $h_i$ of $H$ to $h_i^{-1}$ so that the resulting set has one of the following types: (Here $s$, $t$, $r$, and $s_i$ are always distinct transpositions)

1. $H = \{ss_1,\dotsc, ss_k, s_{k+1},\dotsc, s_{n-2} \}$ with $1 \leq k < n - 2$ such that $Gr_{\{s, s_i\}}$ is a tree.

2. $H = \{s_1s_2, s_3s_4, s_5,\dotsc, s_n\}$ such that $Gr_{\{s, s_i\}}$ is connected and contains exactly $1$ cycle $(E(s_1), E(s_3), E(s_2), E(s_4))$.

3. $H = \{ts_1,\dotsc, ts_k, s, ss_{k+1},\dotsc, ss_{n-3},\}$ with $1 \leq k \leq n - 3$ such that $Gr_{\{t, s, s_i\}}$ is a tree.

4. $H = \{s_1s_2, s_3s_4, s, ss_5,\dotsc, ss_{n-1}\}$ such that $Gr_{\{s, s_i\}}$ is connected and contains exactly $1$ cycle $(E(s_1), E(s_3), E(s_2), E(s_4))$.

5. $H = \{s_1s_2, ts_3,\dotsc,ts_k, s, ss_{k+1},\dotsc,ss_{n-2}\}$ with $3 \leq k \leq n - 2$ such that $Gr_{\{s, t, s_i\}}$ is connected, and contains exactly $1$ cycle $(E(s_1),$ $ E(s), E(s_2), E(t))$.

6. $H = \{s_1s_2, s_3s_4, s, ss_5,\dotsc, ss_{n-1}\}$ such that $Gr_{\{s, s_i\}}$ is connected and contains exactly $1$ cycle $(E(s_1), E(s_3), E(s_2), E(s_4), E(s))$.

7. $H = \{ts_1,\dotsc,ts_k, rs_{k+1},\dotsc,rs_l, s, ss_{l+1},\dotsc,ss_{n-3}\}$ with $1 \leq k < l \leq n-3$ such that $Gr_{\{s, t, r, s_i\}}$ is connected and contains exactly $1$ cycle $(E(s), E(t), E(r))$.
\end{thm}

\begin{proof}
Recall that $G_i$ is the subgroup of $G = S_n$ generated by $\{h_j\}_{j \neq i}$. We divide the problem into $2$ following cases:

$\mathbf{Case\ I}$: $G_i \neq A_n \ \forall i$. In this case, by \cref{lm8}, $Gr_o$ is a tree with $n -2$ edges. Moreover, each $h_i$ is product of $\leq 2$ transpositions, and $E_1(h_i)$, the edge corresponding to the first transposition of $h_i$, lies in $Gr_o$.

Starting with $Gr_o$, if $h_i$ and $h_j$ are not transpositions, and $E_2(h_i) \neq E_2(h_j)$, then we add $E_2(h_i)$ and $E_2(h_j)$ to $Gr_o$. We will have a new graph $Gr(i, j)$ with $n$ edges. Therefore, there must be some cycle $C(i, j)$. If $C(i , j)$ contains some edge $E_1(h_k) = (x, y)$, then by going along the other edge of the cycle, we would get that $x$ and $y$ are in the same $G_k$-orbit, and that cannot happen. Therefore, edges of $C(i, j)$ can only be $E_1(h_i)$, $E_1(h_j)$, $E_2(h_i)$, and $E_2(h_j)$. 

There exists $h_i$ such that it is the product of $2$ transpostions because otherwise the set $\{h_j\}$ cannot generate $S_n$. So now we have $2$ subcases:

$\mathbf{Case\ I.1}$: Suppose that all $h_i$ have the same second transpositions $s$. Then add the edge corresponding to $s$ to $G_o$, we will get a tree $T$. This means that $H$ has type $1$.

$\mathbf{Case\ I.2}$: Suppose that there are $h_i$ and $h_j$ such that $E_2(h_i) \neq E_2(h_j)$. Then consider the graph $Gr(i, j)$, and the resulting cycle $C(i, j)$. WLOG, we can assume that $C(i, j)$ contains both $E_1(h_i)$ and $E_2(h_i)$. As a result, $E_2(h_i)$ must have $2$ vertices that are in different $G_i$-orbit, and therefore, $E_2(h_i) \neq E_2(h_k)$ for any $k \neq i$, so there exists the cycle $C(i, k)$ for any $k \neq i$ that $h_k$ is not a transposition. Now we consider different possible forms of $C(i, j)$:

$\mathbf{Case\ I.2.a}$: If for some $C(i, j)$, $E_1(h_i)$ and $E_2(h_i)$ share no common vertices. Then $C(i, j)$ must be a $4$ cycles and contains $E_1(h_i)$, $E_2(h_i)$, $E_1(h_j)$, and $E_2(h_j)$. Hence, $E_2(h_i)$ must have $2$ vertices that are in different $G_i$-orbits. Similarly, $E_2(h_j)$ also has $2$ vertices that are in different $G_j$-orbits. 

Because vertices in both edges correponding to $h_i$ are in different $G_i$-orbits, $C(i, k)$ must contain both these edges. Therefore, $C(i, k)$ must contain, $E_2(h_k)$, $E_1(h_i)$, $E_2(h_i)$. However, in no way, two vertices of $E_2(h_k)$ can be $2$ of $4$ vertices of $E_1(h_i)$ and $E_2(h_i)$. Contradiction. Therefore, for $k \neq i, j$, $h_k$ must be a transpositions. As a result, $H$ has type $2$.

$\mathbf{Case\ I.2.b}$: If two edges $E_1(h_i)$ and $E_2(h_i)$ in $C(i, j)$ share a common vertex instead. We can again assume that $E_1(h_i) = (1, 2)$ and $E_2(h_i) = (1, 3)$. Then for any $k \neq i$ such that $h_k$ is not a transposition, $E_2(h_k) \neq E_2(h_i)$, and $C(i, k)$ can only have two possible forms:

\begin{center}
\includegraphics[scale=0.45]{alt2.png}
\end{center}

As a result, we can change $h_i$ to $h_i^2 = h_i^{-1}$, and change some other $h_j$ to $h_j^2 = h_j^{-1}$ so that $H$ will have type $1$. \\
        
$\mathbf{Case\ II}$: Assume that there is some $G_i = A_n$. Again, by \cref{lm8}, $h_1$ is a transposition, and other $n-3$ $h_i$'s are product of $2$ transpositions such that $E_1(h_i)$, the edge corresponding to the first transposition of $h_i$, lies in $Gr_o$. Suppose that $h_1 = (x, y)$.

If $E_2(h_i) \neq E_2(h_j)$, and they are different from $E_1(h_1) = (x, y)$, then we add $E_1(h_1)$, $E_2(h_i)$ and $E_2(h_j)$ to $Gr_o$ so that we have a new graph $Gr(i, j)$ with $n$ edges. There must be some cycle $C(i, j)$. $C(i , j)$ cannot contain some edge $E_1(h_k) = (x, y)$ for $k \neq i, j$ and $1$ because otherwise $x$ and $y$ will be in the same $G_k$-orbit by going the other way from $x$ to $y$ along the cycle. So $C(i, j)$ can only contains $\leq 5$ edges: $E_1(h_i), E_1(h_j), E_2(h_i), E_2(h_j), E_1(h_1)$.

Let $M$ be the set of $h_i$ such that both $h_i$ and $h_i^2$ doesn't map $x$ to $y$.  Note that for $h_i \not \in M$ and $h_i \neq h_1$, the second transposition of $h_i$ or $h_i^2$ is exactly $h_1$. Now $M$ must have at least $1$ element because otherwise $H$ cannot generate $S_n$.
 
$\mathbf{Case\ II.0}$: If $|M| = 1$, then by changing $h_i$ to $h_i^2$ appropriately, we will get $H = \{h_1, h_1x_1,\dotsc,h_1x_{n-4}, x_{n-3}x_{n-2}\}$. Moreover, edges corresponding to transpostions $h_1$ and $x_i$ with $i \in \overline{1, n-2}$ will create a forest (and therefore a tree) on the set of vertices $\Sigma$. So $H$ has type $3$.

Now suppose $|M| \geq 2$.

$\mathbf{Case\ II.1}$: Assume that there is some element $h_i$ of $M$ such that $E_1(h_i)$ and $E_2(h_i)$ have no common vertex, and $E_2(h_i) = (u, v)$, with $u$ and $v$ are in different $G_i$-orbit.

For another $h_j \in M$, $E_2(h_j) \neq E_1(h_1)$. Also $E_2(h_i) \neq E_2(h_j), E_1(h_1)$ because of our assumption that $u$ and $v$ are in different $G_i$-orbits. Then we can consider the graph cycle $C(i, j)$ that we defined above. Note that $C(i, j)$ cannot contain only 1 edge of $h_i$ because $2$ vertices of that edge cannot be in the same $G_i$-orbit. Moreover, $C(i, j)$ cannot be a cycle with 3 edges $E_1(h_j), E_2(h_j)$, and $E_1(h_1)$ because $h_j \not \in M$.  As a result, $C(i, j)$ must contain both $E_1(h_i)$ and $E_2(h_i)$ 

Now this cycle $C(i, j)$ has only $4$ possible forms:

\begin{center}
\includegraphics[scale=0.45]{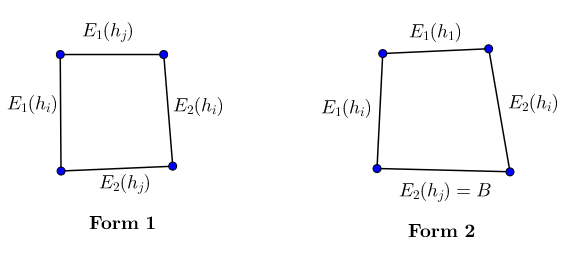}
\end{center}

If $C(i, j)$ has form $1$, then $2$ vertices of $E_2(h_j)$ must be in different $G_j$-orbits. Moreover, if there exists $h_k \in M$ with $k \neq i, j$, then $C(i, k)$ must contains both edges correponding to $h_i$, and also $E_2(h_k)$. Again, similar to $\mathbf{Case\ I.2.a}$, we see that this cannot happen. Hence $M$ must be $\{h_i, h_j\}$, and $H$ has type $4$.

Assume that $C(i, j)$ has form $2$. For $h_k \in M, k \neq i, j$, consider the cycle $C(i, k)$. $C(i, k)$ must contain both edges correponding to $h_i$ and also $E_2(h_k)$. As a result, either $E_2(h_k) = E_2(h_j) = B$ or $C(i, k)$ has the following form:

\begin{center}
\includegraphics[scale=0.45]{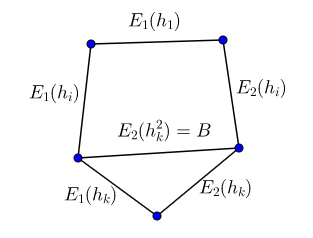}
\end{center}

Therefore, for $h_k \in M$, we can change $h_k$ to $h_k^2 = h_k^{-1}$, if necessary, so that $E_2(h_k) = E_2(h_j) = B$.
So $H$ has type $5$.

Now suppose that $C(i, j)$ has form $3$. Then, by a similar argument we used for form $2$, we can show that $H$ must also have type $5$.

\begin{center}
\includegraphics[scale=0.45]{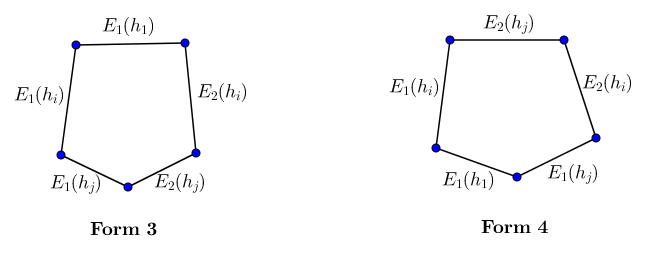}
\end{center}

Finally, we handle the case when $C(i, j)$ has form $4$. If there is any $h_k  \in M$, $k \neq i, j$, then we can consider $C(i, k)$. This cycle again must contain both edges correponding to $h_i$, $E_2(h_k)$, and possibly $E_1(h_1)$. If $C(i, k)$ has $4$ edges, then we can connect $2$ vertices of $E_1(h_j)$ by edges other than $E_1(h_j)$ and $E_2(h_j)$, and therefore these vertices have to be in the same $G_j$-orbit. Contradiction. So $C(i, j)$ must have $5$ edges, and must contain $E_1(h_1)$. But with such arrangment of edges $E_1(h_1)$, $E_1(h_i)$ and $E_2(h_i)$ in form $4$, this cycle must contain all $5$ vertices of $C(i, j)$. Again, we get that two vertices of $E_1(h_1)$ are in the same $G_j$-orbit, and this is impossible. Therefore $M$ must be $\{h_i, h_j\}$. As a result, $H$ must have type $6$. 

$\mathbf{Case\ II.2}$: For each $h_i \in M$, either $E_1(h_i)$ and $E_2(h_i)$ have a common vertex or $E_2(h_i) = (x, y)$, with $x$ and $y$ are in the same $G_i$-orbit. 

Suppose that there exists some $h_i \in M$ such that $E_2(h_i)$ has two vertices in the same $G_i$-orbit. 

If for every $h_j \in M$, $E_2(h_j) = E_2(h_i)$ or $E_2(h_j^2) = E_2(h_i)$, then $H$ must have type $3$.

If there is some $h_j \in M$, both $E_2(h_j)$ and $E_2(h_j^2) \neq E_2(h_i)$, then we can consider the cycle $C(i, j)$. This cycle cannot contain $E_1(h_i)$. Also, either $E_1(h_j)$ is not in $C(i, j)$ or $E_1(h_j)$ and $E_2(h_j)$ have a common vertex. Moreover, this cycle cannot be $(E_2(h_i), E_1(h_j), E_2(h_j))$ because otherwise, $E_2(h_j^2) = E_2(h_i)$

Therefore, the cycle $C(i, j)$ must be either $(E_1(h_1), E_2(h_i), E_2(h_j))$ or $(E_1(h_1), E_2(h_i), E_2(h_j), E_1(h_j))$.

\begin{center}
\includegraphics[scale=0.45]{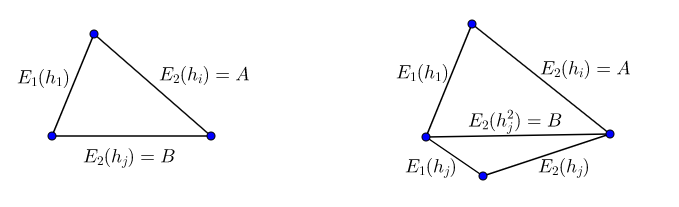}
\end{center}

In each case, for any other $h_k \in M$, then we can change $h_k$ to $h_k^2 = h_k^{-1}$, if necessary, so that all $E_2(h_k)$ can only be the edge $A$ or $B$. As a result, $H$ has type $7$.

We now can assume that for every $h_i \in M$, $E_1(h_i)$ and $E_2(h_i)$ have a common vertex, and $2$ vertices of $E_2(h_i)$ are not in the same $G_i$-orbit. Therefore, for any $h_i, h_j \in M$, $E_2(h_i) \neq E_2(h_j)$, and both are different from $E_1(h_1)$. Moreover, the cycle $C(i, j)$ cannot be a 3 cycles and can only has $1$ of the following forms:

\begin{center}
\includegraphics[scale=0.45]{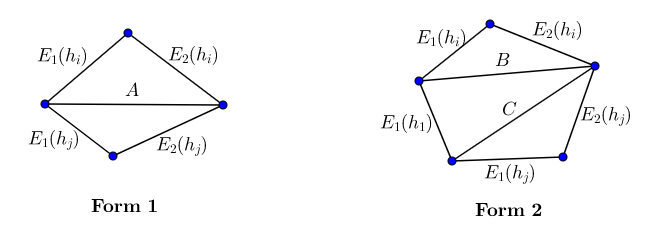}
\end{center}

If for alll $h_j \in M, j \neq i$, $C(i, j)$ has the first form, then we can change $h_k \in M$ to $h_k^2$, such that for all new $h_k$, $E_2(h_k)$ is $A$. Therefore, 
$H$ has type $3$.

If there is some $C(i, j)$ with the second form.  Then, again we can change some $h_k \in M$ to $h_k^2 = h_k^{-1}$ such that $E_2(h_k)$ is either $B$ or $C$. Therefore $H$ has type $7$.

Finally, it is easy to check that each sequence of type $1$ to $7$ is irredundant and generates $S_n$, and so our proof is complete.
\end{proof}

\subsection{Elements in an irredundant generating set of arbitrary length in the symmetric group}

\begin{lem}\label{lm10}
There exists an integer-valued function $f = f(k, q)$ with $2$ variables $q$ and $k \in \mathbb{Z} > 0$ such that $\prod\limits_{i = 0}^k d_i \leq f(k, q)$ for any $\{d_i\}_{i = 1}^k \in \mathbb{Z}$ that satisfies $ \sum\limits_{i=1}^k \dfrac{1}{d_i} = q$ ($q \in \mathbb{Q} > 0$)
\end{lem}
\begin{proof}
For any $q \in \mathbb{Q} > 0$, let $f(1, q) = \floor{\dfrac{1}{q}}\in \mathbb{Z}$. 

We define $f(k, q)$ recursively: $f(k, q) = \max\limits_{\frac{1}{q} < t \in \mathbb{Z} \leq \frac{k}{q}} tf(k-1, q - \frac{1}{t})\ \forall k \geq 2$ and $q > 0 \in \mathbb{Q}$.

Now we prove that $\prod\limits_{i = 0}^k d_i \leq f(k, q)$ for any $\{d_i\}_{i = 1}^k \in \mathbb{Z}$ that satisfies $ \sum\limits_{i=1}^k \dfrac{1}{d_i} = q$ by induction on $k$.

If $k = 1$, it is obvious from the definition of $f(1, q)$. Now suppose $\{d_i\}_{i = 1}^k \in \mathbb{Z}$ that satisfies $ \sum\limits_{i=1}^k \dfrac{1}{d_i} = q$, and WLOG, suppose that $d_1 = \min\limits_{i \in \overline{1, k}} d_i$. Then we get $\dfrac{k}{d_1} \geq q$, and so $d_1 \leq \dfrac{k}{q}$. Suppose $d_1 = t$ with $\dfrac{1}{q} < t \in \mathbb{Z} \leq \dfrac{k}{q}$. Then $\sum\limits_{i = 2}^k d_i = q - \dfrac{1}{t} > 0$, and so by induction hypothesis on $k - 1$, we get $\prod\limits_{i = 2}^k d_i \leq f(k-1, q - \frac{1}{t})$. As a result, $d_1d_2 \dotsc d_k \leq \max\limits_{\frac{1}{q} < t \in \mathbb{Z} \leq \frac{k}{q}} tf(k-1, q - \frac{1}{t}) = f(k, q)$ as desired.
\end{proof}
\begin{defi}
Consider two partitions $\mathbf{P} = \{X_0, X_1,\dotsc, X_m\}$, and $\mathbf{P'} = $\\ 
$\{Y_0, Y_1,\dotsc,Y_{m'}\}$ of $\Sigma = \{1, 2,\dotsc, n\}$ with $m, m' \geq 1$. (Note that these are special partitions because $|X_0|$ and $|Y_0|$ can be zero). We say that $\mathbf{P} < \mathbf{P'}$ (or $\mathbf{P'} > \mathbf{P}$) if one of the following conditions holds:
\begin{itemize}[leftmargin=0.2 in]
	\item $m' = m$ and there exists $i_0$, and $x_0 \in X_0$ such that $Y_0 = X_0 \setminus \{x_0\}$, $Y_{i_0} = X_{i_0} \cup \{x_0\}$, and $X_i = Y_i \ \forall i \neq i_0$. 
	\item $m' = m-1$, and for some $i_0 > 0$, $X_i = Y_i \ \forall i < i_0$, $Y_{i_0} = X_{i_0} \cup X_{i_0+1}$, and $X_{i+1} = Y_i \ \forall i > i_0$.
\end{itemize}
Now for each partition $\mathbf{P}$ of $\Sigma$, let $d(\mathbf{P}) = |X_0| + m \leq 1$. Then for each pair of partitions $\mathbf{P} < \mathbf{P'}$, $d(\mathbf{P'}) = d(\mathbf{P}) - 1$

Therefore, for a partition $\mathbf{P} = \mathbf{P_0}$ of $\Sigma$, any increasing chain of partitions starting at $\mathbf{P_0}$, $\mathbf{P_0} < \mathbf{P_1} < \dotsc. < \mathbf{P_{t-1}}$, has at most $d(\mathbf{P_0})$ elements.
\end{defi}

\begin{lem}\label{lm11}
There exists a function $\phi(k, l)$ so that if $H$ be a subset of $S_n$ ($n \geq 10$), and $\mathbf{P} = \{X_i\}_{i=0}^m$ is a partition of $\Sigma = \{1, 2,\dotsc, n\}$ such that $|X_i| \geq 4\ \forall i \in \overline{1, m}$ and $|X_0| = k \geq 0$ ($X_0$ can be empty), then: 
\begin{itemize}[leftmargin= 0.2 in]
	\item Either there exists $h_0 \in H$ such that there is a partition $\mathbf{P'} > \mathbf{P}$ so that if $G \leq S_n$ has $M$-property wrt $P$ (defined in \cref{def3}), and contains $h_0$, then $G$ also has $M$-property wrt $P'$
	\item Or there exists a subset $K$ of $H$ with $|K| \leq \phi(|X_0|, m)$ such that if some $G \leq S_n$ has $M$-property wrt $\mathbf{P}$ and contains $K$, then $G$ also contains $H$.
\end{itemize}
\end{lem}
\begin{proof}
For $u \in \mathbb{Z} > 0,\ k , l \in \mathbb{Z} \geq 0$, and $\mu$ is a set, we define the following integer-valued functions: 

\vspace{2mm}
$f_1(u, k) = k! \sum\limits_{t = 0}^{k} f(u + t, 1)$, where the function $f$ is defined in \cref{lm10} ,

$g(k, \mu) = k!\sum\limits_{Q = \{Y_1,\dotsc, Y_p\}} \prod\limits_{i = 1}^p f_1(|Y_i|, k)$, where the sum is taken over all partition $Q$ of $\mu$,
\vspace{2mm}

$g_1(k, l) = \max\limits_{\mu \subset \{1,2,\dotsc, l\}}g(k, \mu)$,
\vspace{2mm}

and finally, $\phi(k, l) = (2l + 1)g_1(k, l)$, 
\vspace{2mm}

Now assume by contradiction that the lemma does not hold. 

Consider an element $h_0$ of $H$, and let $G \leq S_n$ be the subgroup that has $M$-property wrt $\mathbf{P}$ (defined in \cref{def3}), and contains $h_0$. 

Write the $M$-decomposition wrt $\mathbf{P}$ (defined in \cref{lm4}) of $h_0$: $h_0 = \alpha.\beta_0.\prod\limits_{i = 1}^{p}\beta_i$ ($\alpha \in \prod\limits_{i = 1}^{m}S_{X_i}, \beta_0 \in S_{X_0}$, and $\beta_i$ commute with each other). Let $\mathbf{Q} = \{Y_1,\dotsc, Y_p\}$ be the partition of $\{1, 2,\dotsc, m'\}$ ($m' \leq m$) that is associated with this decomposition so that $\beta_i \not\in S_{X_t} \ \forall t \in \overline{0, m}$ is a single cycle that contains all elements of $X_j$ that are not fixed by $h$ with $j \in Y_i$.

Now we prove that for each $1 \leq j \leq m'$, elements in $X_j$ lies in the cycle $\beta_i$ with $j \in Y_i$ in such a way that the distance between 2 consecutive elements is the same. Assume by contradiction that this is not the case for some $X_{j_0}$, then if we let $d_0$ be the smallest distance between 2 elements of $X_{j_0}$ that lies in the cycle $\beta_{i_0}$, we will have $\beta_{i_0}^{d_0} = Z.T = (y_1, y_2, \dotsc, y_l, a, b, x, x_1, x_2, \dotsc, x_k).T$, where $a, b \in X_{j_0}$ while $x \in X_{j_1}, j_1 \neq j_0$. Here $Z$ and $T$ are products of distinct terms in the cycle decomposition of $\beta_{i_0}^{d_0}$.

First, we have $h_0 = \alpha.\prod\limits_{i = 0}^{p}\beta_i \in G$, so $\prod\limits_{i = 0}^{p}\beta_i.\alpha \in G$. 

Therefore, $\alpha.\prod\limits_{i = 0}^{p}\beta_i.\prod\limits_{i = 0}^{p}\beta_i.\alpha = \alpha\prod\limits_{i = 0}^{p}\beta_i^2\alpha \in G$. Again, we get $\alpha^2\prod\limits_{i = 0}^{p}\beta_i^2 \in G$. In a similar manner,we will get $L = \alpha^{d_0}\prod\limits_{i = 0}^{p}\beta_i^{d_0} \in G$. 

Hence, 
\begin{equation*}
\begin{split}
L_1 &= (a, b)\alpha^{d_0}\prod_{i = 0}^{p}\beta_i^{d_0}(a, b) \\
	& = (a, b)\alpha^{d_0}(a, b).(y_1, y_2, \dotsc, y_l, b, a, x, x_1, x_2, \dotsc, x_k)T. \prod_{i \neq i_0}\beta_i^{d_0} \in G.
\end{split}
\end{equation*}	

As a result, \begin{equation*}
\begin{split}
LL_1^{-1} & = \alpha^{d_0}(y_1, y_2, \dotsc, y_l, a, b, x, x_1, x_2, \dotsc, x_k) \cdot \\
& (x_k, \dotsc, x_2, x_1, x, a, b, y_l, \dotsc, y_2, y_1)(a, b)\alpha^{-d_0}(a, b) \\ 
				& = \alpha^{d_0}(x, b, a)(a, b)\alpha^{-d_0}(a, b) \in G
\end{split}
\end{equation*}

Hence, $(a, b)\alpha^{-d_0}(a, b)\alpha^{d_0}(x, b, a) = \alpha_1 (x, b, a) \in G$ with $\alpha_1 \in S_{X_{j_0}}$. 

If $j_1 > 0$, then let $\Sigma' = X_{j_0} \cup X_{j_1}$. Applying \cref{lm5} for the partition $\{\Sigma \setminus \Sigma', X_{j_0}, X_{j_1}\}$, we get $A_{\Sigma'} \subset G$. For $\sigma \in S_{\Sigma'}$ and $h \in G$, if $\sigma$ is an even permutation, $\sigma h \sigma^{-1} \in G$ because $A_{\Sigma'} \subset G$. Otherwise, if $\sigma$ is an odd permutation, $\sigma h \sigma^{-1} = \sigma(a, b)(a, b)h(a, b)(a, b)\sigma^{-1} \in G$ because $(a, b)h(a, b) \in G$, and $\sigma(a, b) \in A_{\Sigma'} \subset G$. Therefore, $G$ has $M$-property wrt $\mathbf{P'} = \{X_i, X_{j_0} \cup X_{j_1}\}_{i \neq j_0, j_1}$. 

Now if $j_1 = 0$ instead, then we get $A_{\{x\} \cup X_{j_0}} \subset G$. By a similar argument as above, we get $\sigma h \sigma^{-1} \in G\ \forall h \in G, \sigma \in S_{\{x\} \cup X_{j_0}}$. Hence, $G$ has $M$-property wrt $\mathbf{P'} = \{X_0\setminus\{x\}, \{x\} \cup X_{j_0}, X_i \}_{i \neq j_0 > 0}$. 

In either cases, $\mathbf{P'} > \mathbf{P}$, and the first statement in the lemma holds, and this contradicts with our initial assumption. 

In a similar manner, we can prove that all elements of $X_j$ with $j \in Y_i$ appear in the cycle $\beta_i$ by using the identity:
\[(y_1, y_2, \dotsc, y_l, a, x, x_1, x_2, \dotsc, x_k)(x_k \dotsc, x_2, x_1, x, b, y_l \cdots y_2, y_1) = (x, b, a)\], for $a, b \in X_j$, and $x \in X_{j_1} \neq X_j$. \\

As a result, for all $h \in H$, the $M$-decomposition of $h$ is $h = \alpha.\beta_0.\prod\limits_{i=1}^{p}\beta_i$ with the associated partition $\mathbf{Q} = \{Y_1,\dotsc, Y_p\}$ of $\{1, 2, \dotsc, m'\} (m' \leq m)$ such that all elements of $X_j$ appear in $\beta_i$, and the distance between any $2$ elements of $X_j$ in this cycle is the same and is $d_j\ \forall j \in Y_i$. 

Then there exists $h'$, which is a conjugate of $h$ by an element in $\prod\limits_{i=1}^m S_{X_i}$, such that the cycle $\beta_i$ ($1 \leq i \leq p$) in the $M$-decomposition of $h'$ contains all elements $X_j$ with $j \in Y_i$ which are in an increasing order and the distance between any $2$ elements in $X_j$ is the same and is $d_j$ ($j \in Y_i$). Therefore, if $\beta_i$ contains $k$ elements and $t$ elements in $X_0$, then $k = t + \sum\limits_{j \in Y_i}\frac{k}{d_j}$, $t \leq |X_0|$. Therefore, we get $1 = \frac{1}{k} + \dotsc +  \frac{1}{k} + \sum\limits_{j \in Y_i}\frac{1}{d_j}$ (with $t$ terms $\frac{1}{k}$). By \cref{lm10}, for $t \in \overline{1, |X_0|}$, $\prod\limits_{j \in Y_i} d_j \leq k^t\prod\limits_{j \in Y_i} d_j \leq f(|Y_i| + t, 1)$. 

Notice that the first element of $X_j$ that is in the cycle $\beta_i$ must appear in the $l$th position for some $l \in \overline {1, d_j}$, and subsequent element of $X_j$ in $\beta_i$ is determined by adding a multiple of $d_j$ to that first element. Therefore, $\prod\limits_{i = 0}^p\beta_i$ of $h'$ has at most 
\begin{equation*}
|X_0|! \cdot \sum_{k = t + \sum_{j \in Y_i}\frac{k}{d_j}} \prod_{j \in Y_i}{d_j} \leq |X_0|! \sum\limits_{t = 0}^{X_0} f(|Y_i| + t, 1) = f_1(|Y_i|, |X_0|)
\end{equation*}
possible forms. 

Therefore, $h' = \alpha'.T$ where $\alpha' \in \prod_{i=1}^mS_{X_i}$ so that $T = $ can only have $\leq |X_0|!\sum\limits_{Q = \{Y_1,\dotsc, Y_p\}} \prod\limits_{i = 1}^p f_1(|Y_i|, |X_0|) = g(X_0, \mu) \leq \max\limits_{\mu \subset \{1,2,\dotsc, m\}}g(X_0, \mu) = g_1(X_0, m)$ possible forms, where $Q$ is a partition of some subset $\mu$ of $\{1, 2,\dotsc, m\}$. \\

Now for $h, k \in H$, we say that $h \sim k$ if the correponding $h' = \alpha_h.T$ and $k' = \alpha_k.T$ for some $\alpha_h$ and $\alpha_k \in \prod_{i=1}^{m}S_{X_i}$.  So we have at most $g_1(X_0, m)$ equivalent classes in $H$. 

Consider a particular equivalent class $C$ and $h_0 \in C$. Suppose $G$ has $M$-property wrt $\mathbf{P}$, and $h_0, k \in G$ for some $k \in C$, then $\alpha_{h_0}\alpha_k^{-1} =  \prod_{i \in J}x_i \in \prod_{i \in J}S_{X_i} \in G$ with $x_i \neq 1$. Then $A_{X_i} \subset G$ for all $i \in J$ . Let $J_0 = \{i_0 \text{ such that } \exists k \text{ so that } \alpha_{h_0}\alpha_k^{-1} =  \prod_{i=1}^{m}x_i \in \prod_{i=1}^{m}S_{X_i} \text{ so that } x_{i_0} \neq 1\}$. Then $G$ needs to contain at most $m$ elements $k \neq h_0 \in C$ so that $A_\{X_i\} \subset G\ \forall i \in J_0$. Because elements $x_i$ from the product $\alpha_{h_0}\alpha_k^{-1} =  \prod_{i=1}^{m}x_i \in \prod_{i=1}^{m}S_{X_i}$ can be odd permutations, we will need at most other $m$ elements $k' \in C$ so that if $G$ contains these $k$ and $k'$ together with $h_0$, then $G$ will contain $C$ . Therefore, if $G$ is some subgroup that has $M$-property wrt $\mathbf{P}$, then $G$ just needs to contain at most $(2m + 1)$ elements in an equivalent class $C$ to contain all elements in this class.

As a result, there exists $K \subset H$ $|K| \leq (2m + 1)g_1(X_0, m) = \phi(X_0, m)$ such that if $G$ contains $K$, then $G$ contains $H$.
\end{proof}

\begin{lem}\label{lm12}(Generalization of \cref{lm6})
There exists a function $\psi = \psi(k, l)$ such that for any partition $\mathbf{P} = \{X_0, X_1,\dotsc, X_m\}$ of $\{1,2,\dotsc, n\}$ with $|X_i| \geq 4 \ \forall i = \overline{1, m}$, and any subset $H \subset S_n$, there exists a subset $K \subset H$ such that $|K| \leq \psi(|X_0|, m)$ and that any subgroup $G$ with $M$-property wrt $\mathbf{P}$ that contains $K$ also contains $H$.
\end{lem}
\begin{proof}
Consider $\psi(k, l) = \phi(k, l) + k + l$, where the function $\phi$ is defined in \cref{lm11}

Let $\mathbf{P_0} = \mathbf{P}$, and start with an empty set $K_0$. For $i \geq 0$, if we have partition $\mathbf{P_i}$ and there exists some $h_i \in H$ such that any group $G$ that has $M$-property wrt $\mathbf{P_i}$ containing $h_i$ also has $M$-property wrt $\mathbf{P_{i+1}}$ with some $\mathbf{P_{i+1}} > \mathbf{P_i}$, we will add that $h_i$ to $K_0$. Because the longest chain of increasing paritions has size $|X_0| + m$, $|K_0| \leq |X_0| + m$. 

Suppose that $G$ is some subgroup of $S_n$ that has $M$-property wrt $\mathbf{P_{|K_0|}}$. Because $\mathbf{P_{|K_0|}}$ is the final partition obtained from the previous process, there is no partition $\mathbf{Q} = \{Q_0, Q_1, \dotsc, Q_{m'}\} > \mathbf{P_{|K_0|}}$ of $\{1, 2, \dotsc, n\}$ such that there is some $h \in H$ so that any group $G$ that has $M$-property wrt $\mathbf{P_{|K_0|}}$ containing $h$ also has $M$-property wrt $\mathbf{Q}$.

Therefore, by \cref{lm11}, there exist $K_1 \leq H$ with $|K_1| \leq \phi(|X_0|, m)$ so that any $G$ that has $M$-property wrt $P_{|K_0|}$ and contains $K_1$ also contains $H$. Now consider any $G \leq S_n$ containing $K_0 \cup K_1$ and has $M$-property wrt $P$. Because $K_0 \subset G$, $G$ must have $M$-property wrt $P_{|K_0|}$. Now because $G$ contains $K_1$, $G$ must contain $H$. Finally, we have $|K_0 \cup K_1| \leq |K_0| + |K_1| \leq |X_0| + m + \phi(|X_0|, m) = \psi(|X_0|, m)$, so we can finish the proof here.
\end{proof}

\begin{defi}
For a partition $\mathbf{P}= \{X_0, X_1,\dotsc, X_m\}$ of $\{1, 2,\dotsc, n\}$. A subgroup $I \leq S_n$ has $\mathbf{SP}$-property wrt $\mathbf{P}$ if given any $m$ transpositions $s_i \in S_{X_i}$ for $i \in \overline{1, m}$, for any $s \in \prod\limits_{i = 1}^m S_{X_i}$, there exists $\epsilon_t \in \{0, 1\}$ so that $s \cdot \prod_{t = 1}^m s_t ^{\epsilon_t} \in I$. 

It is obvious that if $I'$ has $\mathbf{SP}$-property wrt $\mathbf{P}$, and $I' \leq I \leq S_n$, then $I$ also has $\mathbf{SP}$-property wrt $\mathbf{P}$

Given a subgroup $I \leq S_n$ with $\mathbf{SP}$-property wrt $\mathbf{P}$, we define $G(I, K, \mathbf{P})$ to be the subgroup of $S_n$ that is generated by $xkx^{-1}$ with $x \in I$, $k \in K$.
\end{defi}

\begin{lem}\label{lm13}
Assume $\mathbf{P} = \{X_0, X_1,\dotsc, X_m\}$, $m \geq 1$, is a partition of $\{1, 2,\dotsc, n\}$, and $X$ is a subset of $\cup_{i = 1}^m X_i$ such that $|X| = k$, and $|X_i \setminus X| \geq 4\ \forall i \in \overline{1, m}$. Moreover, assume that $I \leq S_n$ has $\mathbf{SP}$-property wrt $\mathbf{P}$. Then there exists $K \subset H$, with $|K| \leq k(m + |X_0|)$ such that there exists $\{g_h\}_{h \in H} \subset G(I, K, \mathbf{P})$ so that $\forall h \in H$, $g_hh$ fixes all points in $X$.
\end{lem}
\begin{proof}
We prove by induction on $k \geq 0$. 

For $k = 0$, there is nothing to prove.

Suppose the statement is true for $k -1$ ($k \geq 1$). We prove that it is also true for $k$. Consider a point $x \in X$. Applying the induction hypothesis on $k - 1$ for $X \setminus \{x\}$, there exists $K_1 \subset H$, with $|K_ 1| \leq (k - 1)(m + |X_0|)$ such that there exists $l_h = g_hh$ fixed all points in $X \setminus \{x\}$. Here $\{g_h\}_{h \in H} \subset G(I, K_1, \mathbf{P})$

Consider the set $S = \{l_h(x), h \in H, l_h(x) \neq x\}$, which is the set of points $\neq x$ that $l_h$ maps $x$ to. For $s, t \in S$, we say that $s \sim t$ if $s$ and $t$ are in the same $X_i$ for some $i \geq 1$. Then we have $t \leq m + |X_0|$ equivalent classes in $S$. 

Now suppose that $\{s_i\}_{i = 1}^t$ are representatives of these equivalent classes such that $l_{k_i}(x) = s_i$. Let $K_2 = \{k_i\}_{i = 1}^t \subset H$. We will prove that $K = K_1 \cup K_2$ is the set $K$ that we need to find. First notice that $|K| \leq |K_1| + |K_2| \leq (k-1)(m + |X_0|) + (m + |X_0|) = k(m + |X_0|)$

For $h \in H$ such that $l_h(x) = x$, then $l_h(x)$ fixes all points in $X$, and we have $g_hh = l_h$ fixes all points in $X$ with $g_h \in G(I, K_1, \mathbf{P}) \leq G(I, K, \mathbf{P})$.

We now prove that for other $h \in H$, there exists $m_h = gl_{k_i}g^{-1}$ for some $i \in \overline {1, t}$, and some $g \in I$ such that $m_h^{-1}l_h$ fixes all points in $X$. Consider $h \in H$ such that $l_h(x) \neq x$. Assume that $l_h(x)= s_h \sim s_i$ for some $i \in \overline{1, t}$. If $s_h = s_i$, then choose $m_h = l_{k_i}$, and we will have $m_h^{-1}l_h$ fixes $x$ and all points in $X \setminus \{x\}$. 

We suppose that $s_h \neq s_i$ are in some same set $X_j$ ($1 \leq j \leq m$). First, $s_h$ and $s_i \not \in X \setminus \{x\}$ because $l_h$ and $l_{k_i}$ don't fixes $s_h$ and $s_j$. 

Now choose two other elements $y_j \neq z_j \neq s_h,\ s_i \in X_j \setminus X$, and $y_t \neq z_t \in X_t \setminus X$ for $1 \leq t \neq j \leq m$ ($y_j$, $z_j$, $y_t$, and $z_t$ exist because $|X_i \setminus X| \geq 4$). By definition, $n_h = (s_hs_i) \cdot \prod\limits_{t= 1}^m (y_tz_t)^{\epsilon_t} \in I$ for some $\epsilon_t \in \{0, 1\}$. As a result, $m_h = n_hl_{k_i}n_h^{-1} = n_hl_{k_i}n_h$ will still fixes all points in $X \setminus \{x\}$ because $s_h, s_i \not \in X \setminus \{x\}$, and $y_t, z_t \not \in X \setminus \{x\}$ for $t \in \overline{1, m}$. Furthermore, $m_h$ maps $x$ to $s_h$, and therefore $m_h^{-1}l_h$ fixes all points in $X$.

Now note that we have $m_h^{-1}l_h = m_h^{-1}g_h h = u_hh$ fixes all points in $X$. Moreover, $u_h = m_h^{-1}g_h = n_hl_{k_i}^{-1}n_hg_h = n_hg_{k_i}n_h^{-1}n_hk_in_h^{-1}g_h \in G(I, K, \mathbf{P})$ because $n_hg_{k_i}n_h^{-1}$ and $g_h \in G(I, K_1, \mathbf{P}) \subset G(I, K, \mathbf{P})$, and $n_hk_in_h^{-1} \in G(I, K_2, \mathbf{P}) \subset G(I, K, \mathbf{P})$. Therefore, $K$ is the set that we need to find to finish the induction step on $k$ and, therefore, to finish the proof.
\end{proof}

\begin{defi}\label{def9}
Let $\mathbf{P}= \{X_0, X_1,\dotsc, X_m\}$ be a partition of $\{1, 2,\dotsc, n\}$, and $I$ be a subgroup of $S_n$ that has $\mathbf{SP}$-property wrt $\mathbf{P}$. We say that $G \leq S_n$ has $\mathbf{K}$-property wrt $\mathbf{P}$ and $I$ if $xgx^{-1} \in G \ \forall g \in G, x \in I$.  
\end{defi}

\begin{lem}\label{lm14}
There exists a function $\omega = \omega(u, v)$ so that for any partition $\mathbf{P}= \{X_0, X_1,\dotsc, X_m\}$ of $\{1, 2,\dotsc, n\}$ with $|X_i| \geq 6 \ \forall i \in \overline{1, m}$, any subgroup $I \leq S_n$ that satisfies $\mathbf{SP}$-property wrt $\mathbf{P}$, and subset $H \subset S_n$, there exists a subset $K \subset H$ such that $|K| \leq \omega(|X_0|, m)$ and that any subgroup $G$ with $\mathbf{K}$-property wrt $\mathbf{P}$ and $I$ that contains $K$ also contains $H$.
\end{lem}
\begin{proof}
By \cref{lm12}, there exists a function $\psi(k, l)$ so that for any partition $\mathbf{P'} = \{X_0, X_1, \dotsc, X_{m'} \}$ of $\{1, 2,\dotsc, n\}$ with $|X_i| \geq 4$, and any subset $U \subset S_n$, there exists a subset $V \subset U$ with $|M| \leq \psi(|X_0|, m)$ such that any subgroup $G \leq S_n$ that has $M$-property wrt $\mathbf{P}$ and contains $V$ also contains $U$. Then we define $\omega(k, l) = \psi(k, l) + 2l(k + l)$ for $k, l \in \mathbb{Z} \geq 0$.

Now consider any subset $H \subset S_n$. 

Let $X = \cup_{i=1}^m\{x_i, y_i\}$, in which $x_i$ and $y_i$ are two fixed elements of $X_i$. By \cref{lm13} for $H$, $X$, and partition $\mathbf{P}$, there exists $K_1 \subset H$ with $|K_1| \leq |X| \cdot (m + |X_0|) = 2m(m + |X_0|)$ such that $\forall h \in H, g_hh$ fixed all points in $X$ for some $g_h \in G(I, K_1, \mathbf{P})$.

Let $H' = \{g_hh, h \in H\}$. Then by definition of $\psi$, there exists a subset $K' \subset H'$ with $|K'| \leq \psi(X_0, m)$ so that any subgroup $G \leq S_n$ with $\mathbf{M}$-property wrt $\mathbf{P}$ that contains $K'$ also contains $H'$. We can assume that $K' = \{g_{k_i}k_i\}_{i = 1}^t$ for some $k_i \in H$ and $t \leq \psi(|X_0|, m)$.

Now let $K = K_1 \cup \{k_i\}_{i = 1}^t$. We will prove that $K$ is the subset of $H$ that we need to find.

Suppose that $G'$ is the subgroup of $S_n$ that is generated by $gk'g^{-1}$ with $k' \in K'$ and $g' \in \prod\limits_{i=1}^mS_{X_i}$. Then $G'$ has $M$-property wrt $\mathbf{P}$ and also contains $K'$. Therefore, $G'$ contains $H'$. 

Now for each generator $l = g_lg_{k_i}k_ig_l^{-1}$ of $G'$ with $g_l \in \prod_{j = 1}^mS_{X_i}$. By definition of $I$, there exists $\epsilon_t \in \{0, 1\}$ such that $n_l= g_l \cdot \prod\limits_{t = 1}^m(x_ty_t)^{\epsilon_t} \in I$.

Because $g_{k_i}k_i \in H$ fixes $x_j$ and $y_j$, 
\begin{equation*}
n_lg_{k_i}k_in_l^{-1} = g_l(\prod\limits_{t = 1}^m(x_ty_t)^{\epsilon_t}g_{k_i}k_i\prod\limits_{t = 1}^m(x_ty_t)^{\epsilon_t})g_l^{-1}= g_l(g_{k_i}k_i)g_l^{-1} = l
\end{equation*}

As a result, $l \in G(I, K, \mathbf{P})$ because $g_{k_i} \in G(I, K_1, \mathbf{P}) \subset G(I, K, \mathbf{P})$, and $k_i \in K \subset G(I, K, \mathbf{P})$. Therefore, $G' \leq G(I, K, \mathbf{P})$, and so $H' \subset G(I, K, \mathbf{P})$. Then for any $h \in H$, $g_hh \in H' \subset G(I, K, \mathbf{P})$, and because $g_h \in G(I, K_1, \mathbf{P}) \subset G(I, K, \mathbf{P})$, $h \in G(I, K, \mathbf{P})$. Hence $H \subset G(I, K, \mathbf{P})$.

Finally, any subgroup of $S_n$ with $\mathbf{K}$-property wrt $\mathbf{P}$ and $I$ that contains $K$ also contains $G(I, K, \mathbf{P})$ and, therefore, contains $H$. Because $|K| \leq |K_1| + t \leq 2m(m + |X_0|) + \psi(|X_0|, m) = \omega(|X_0|, m)$, we finish our proof here. 
\end{proof}

\begin{lem}\label{lm15}
Suppose that $H \leq S_2 \wr S_m$, and $\forall n \geq n_0$, if $K \neq A_n \leq S_n$, then $m(K) \leq n - k + 1$. Then for all $m \geq \max\{n_0, k + 1\}$, $m(H) \leq 2m - k$
\end{lem}
\begin{proof}
By \hyperref[lm1]{Whiston}'s lemma, we can assume that there is an irredundant generating set $\{g_1,\dotsc, g_l\}$ with $l = m(H)$ such that $\{g_1,\dotsc, g_r\}$ generates block permutation, and $\{g_{r+1},\dotsc, g_l\} \subset S_2^m$.

$\mathbf{Case\ 1}$: The subgroup generating block permutation and generated by $\{g_1,\dotsc, g_r\}$ doesn't act as $S_m$ or $A_m$ on $m$ blocks, then $r \leq m - k  + 1$. 

If $K = \{g_{r+1}, \dotsc, g_l\}$ doesn't generate the whole vector space $S_2^m = F_2^m$, $l-r \leq m - 1$. Hence, $m(H) = l \leq m - k + 1 + m - 1 = 2m - k$. 

Now suppose $K$ generates the whole vector space $F_2^m$. Take some $g_i$ that is not $(1,1,\dotsc, 1)$. Then there is some entry 0 of $g_i$, and WLOG, assume that there is one entry 0 on $g_i$'s last coordinate. Also, WLOG, assume that $1$ appears in the first coordinate of $g_i$. Because $H$ is transitive, there exists $h$ generated by $g_1,\dotsc, g_r$ that moves $m$th block to the first block. Hence $h^{-1}g_ih$ will have the last entry turned to $1$ and therefore is of the form $\sum \alpha_jg_j$ with some $\alpha_j \neq 0$ and $j \neq i$. Hence $\{g_1, \dotsc, \hat{g_j},..,g_l\}$ generates the whole space. This is impossible because the original set $\{g_1,\dotsc,g_l\}$ must be irredundant.

$\mathbf{Case\ 2}$: The subgroup generating block permutation and generated by $\{g_1,\dotsc, g_r\}$ acts as $S_m$ or $A_m$ on $m$ blocks.

Then, since $F_2$ is abelian, for any element $g_i$, after performing an even permutation on the coordinates of $g_i$, the new vector will be generated by $\{g_1,\dotsc, g_k\}$ and $g_i$. Now, WLOG, suppose that the element $g_{k+1}$ of $\{g_{k+1},\dotsc, g_l\} = (1,,\dotsc,1,0,\dotsc,0) \neq (1,1,,\dotsc,1)$. First permute $2$ zero entries or $2$ one entries, if any, of $g_{k+1}$, and then move the first entry, which is $1$, to the position of the first zero entry. Hence, we get the new vector of the form $(0,1,\dotsc1,0,\dotsc,0)$ that is generated by $\{g_1,\dotsc,g_{k+1}\}$ . Continue doing this until we get the set of vectors $(0,..0,1,..1,0,..,0)$ with $i$ first zero entries for any $i$. Therefore, by using only $\{g_1,\dotsc,g_k\}$, and $g_{k+1}$, we can already cover the subspace $\{(x_1,\dotsc,x_m): x_1+\dotsc+x_m = 0\}$ of $F_2^m$. To generate the whole space, we only need at most $1$ more vector. So, in total, we only need $\leq 2$ elements $g_i$ together with $\{g_1,\dotsc,g_k\}$ to generate all elements in $K$. Therefore $m(H) = l \leq m - 1 + 2 = m + 1 \leq 2m - k$  for $m \geq k + 1$.  
\end{proof}

\begin{nt}\label{nt2}
For any direct product $G = \prod_{i = 1}^k G_i$ for some finite group $G_i$, $\pi_i: G \to G_i$ denotes the projection map from $G$ to its factor $G_i$, and $\pi_{\{i_1, i_2,\dotsc, i_j\}}: G \to \prod\limits_{t = 1}^j G_{i_t}$ denotes the projection map from $G$ to the product of its factors $G_{i_t}$ for $t = \overline{1, j}$. Note that $\pi_i = \pi_{\{i\}}$. 

In particular, for $g = \prod\limits_{t = 1}^k g_t \in G$ with $g_i \in G_i$, $\pi_i(g) = g_i$, and\\ $\pi_{\{i_1, i_2, \dotsc, i_j\}}(g) = \prod\limits_{t = 1}^jg_{i_t}$. 

\vspace{2mm}
If $H$ is a subset or subgroup of $G$, $\pi_i(H)$ denotes the image of $H$ under the projection map $\pi_i$. Similarly, we also have $\pi_{i_1, i_2, \dotsc, i_j}(H)$.
\end{nt}
\begin{lem}\label{lm16}
Let $H$ be a subgroup of $G = \prod\limits_{i = 1}^k G_i$ for some finite group $G_i$, and let $m(H) = m$. Assume also that $\pi_i: G \to G_i$ be the projection map from $G$ to its factor $G_i$. Then there exists $H_i \subset \prod\limits_{j = i}^k G_j \cap H$, so that $m = \sum\limits_{i = 1}^k|H_i| \leq \sum\limits_{i = 1}^k m(T_i)$, $T_i = \pi_i(\genby{H_i})$. ($\pi_i$ is defined in \cref{nt2})
\end{lem}
\begin{proof}
Suppose $K = \{h_1,\dotsc,h_m\}$ is an irredundant generating set of $H$.
Let $H_1$ be the smallest subset of $K_0 = K = \{h_1,\dotsc,h_m\}$ such that $\pi_1(H_1)$ generates $T_1 = \pi_1(H)$. Then $\pi_1(H_1)$ is an irredundant generating set of $T_1$, and so $|H_1| \leq m(T_1)$. Then there exists $g_i$ generated by $H_1$ such that $h_i' = g_ih_i$ with $i > r$ and $\pi_1(h_i') = 1$. As a result, $\genby{K_1} = \genby{\{h_i', i >r\}} \leq \prod\limits_{i = 2}^k G_i$. The set $K_1$ is irredundant because $K$ is irredundant. 

Now take $H_2$ to be the smallest subset of $K_1$ that generates $T_2 = \pi_2(\genby{K_1}) \leq G_2$. So $\pi_2(\genby{K_1}) = \pi_2(\genby{H_2})$. Then we will again have an irredundant set $K_2$ so that $\pi_1(K_2) = \pi_2(K_2) = 1$. We continue this process to build $H_i \subset K_{i-1}$ that generates $\pi_i(\genby{K_{i-1}}) \leq G_i$. From this, we will have an irredundant set $K_i$ that are built from remaining $h_i$: $K_i = \{t_jh_j, t_j \in \genby{X_1 \cup \dotsc \cup X_i}$ that fixes first $i$ factors or $\pi_1(K_i) = \dotsc = \pi_i(K_i) = 1 \}$. Since $K_i$ must be irredundant, $K_k$ must be an empty set. As a result, we have $m = \sum\limits_{i = 1}^k|X_i| \leq \sum\limits_{i = 1}^k m(\pi_i(\genby{H_i})) = \sum\limits_{i = 1}^k m(T_i)$. 
\end{proof}

Now using previous results, we will try to characterize elements that can be in an irredundant generating set of length $n - 3, n - 4$, or $n - 5$ in $S_n$

\begin{thm}\label{thm7}
For $k \leq 7$, there exists function $\Psi$ such that $m(H) = m \leq n -k$ for any transitive subgroup $H$ of $S_n$ that is not either $A_n$ or $S_n$ with $n \geq \Psi(k)$
\end{thm}
\begin{proof}
Now we will define $\Psi(k)$ recursively:

$\Psi(1) = 1$.

$\Psi(k) =\max\{2\Psi(k-1) \cdot (k -1),\ (k + \max\limits_{\substack{x \leq \Psi(k-1) \\ y \leq 3}} \omega(x, y)) \cdot (k -1), n_0, k^2\}$ for $k \geq 2$, where $\omega$ is defined in \cref{lm14}

First, by \cref{lm3}, for $n \geq \Psi(k) \geq n_0$, $m(H) \leq \frac{n}{2} \leq n - k$ ($n \geq k^2 \geq 2k$) if $H$ is primitive subgroup of $S_n$. So we only need to handle the case when $H$ is an imprimitive but transitive subgroup of $S_n$. Therefore, we can assume that $H \leq S_{\Gamma} \wr S_{\Delta}$ ($n = |\Gamma||\Delta|, |\Gamma|, |\Delta| \geq 2$)

By \hyperref[lm1]{Whiston}'s lemma, if $\{g_1, g_2,\dotsc, g_m\}$ is an irredundant generating set of $H$, then we can asssume that $\{g_1,\dotsc,g_l\}$ generates block permutation while $\{g_{l+1},\dotsc,g_m\}$ is an irredundant set in $\prod\limits_{i = 1}^{|\Delta|}S_{\Gamma}$. So $l \leq |\Delta| - 1$.

We will prove that for $n \geq \Psi(k), m(H) = m \leq n -k$ by induction on $1 \leq k \in \mathbb{Z} \leq 9$. 

If $k = 1$, then it is true from Whiston's paper. Suppose that the statement is true for $k-1$, we prove that it is true for $k$ ($k \leq 7$)

Consider $n \geq \Psi(k)$. Suppose by contradiction that we have an imprimitive but transitive subgroup $H \leq S_{\Gamma} \wr S_{\Delta}$ so that $n = |\Gamma||\Delta|$, and $m(H) > n - k$. 

By \cref{lm16} for $K = \genby{\{g_{l+1},\dotsc, g_m\}} \leq \prod_{i = 1}^{|\Delta|}S_{\Gamma}$, there exists $K_i \subset \prod\limits_{i = 1}^k G_j \cap K$.  so that $m - l \leq \sum\limits_{j=i}^{|\Delta|} m(T_i)$ with $T_i = \pi_i(\genby{K_i}) \leq S_{\Gamma}$. Here $\pi_i$ is the projection from $K$ to $i$th copy of $S_{\Gamma}$ (defined in \cref{nt2})

If $K$ acts as $S_{\Gamma}$ or $A_{\Gamma}$ on $\Gamma_i$ for some $i \in \overline{1, |\Delta|}$. Then, by \cref{lm2}, we have $m(H) \leq |\Gamma| + 2 |\Delta| - 3 = n - (|\Gamma| - 2)(|\Delta| - 1) - 1 \leq n - (k - 1) - 1 = n - k$ for $n = |\Gamma||\Delta| \geq k^2 - 1$ and $|\Gamma| > 2$.

If  $|\Gamma| = 2$ then $|\Delta| \geq \Psi(k-1)$, and by \cref{lm15}, $m(H) \leq n - k$.

Therefore, we can suppose that $K$ cannot acts as $S_{\Gamma}$ or $A_{\Gamma}$ on $\Gamma_i$, or $\pi_i(K) \neq S_{\Gamma}$ or $A_{\Gamma} \ \forall i \in \overline{1, |\Delta|}$. Then $T_i \leq \pi_i(K) \neq S_{\Gamma}$, and, therefore, $m(T_i) \leq |\Gamma| - 2$ because of the strongly flatness of $S_{\Gamma}$.

As a result, $m(H) = m = l + m - l \leq |\Delta| - 1 + |\Delta|(|\Gamma| - 2) = n - |\Delta| - 1 \leq n - k$ if $|\Delta| \geq k - 1$. So we must have $2 \leq |\Delta| < k - 1$. As a result, $|\Gamma| \geq 2\Psi(k - 1)$, and $|\Gamma| \geq \max\limits_{\substack{x \leq \Psi(k-1) \\ y \leq 3}} \omega(x, y) + k$

\vspace{2mm}

\begin{lem}\label{lm17}
For $i \in \overline{1, |\Delta|}$, $m(T_i) \leq |\Gamma| - 4$ or $T_{i}$, and therefore, $\pi_{i}(K)$ has $\mathbf{SP}$-property (defined in \cref{def9}) wrt some partition $\mathbf{P} = \{P_0, P_1,\dotsc, P_k\}$ of $\zeta = \{1, 2, \dotsc, |\Gamma| \}$ with $|P_0| < \Psi(k - 1)$ and $k \leq 3$.
\end{lem}
\begin{proof}
Suppose that for some $i_0 \in \overline{1, |\Delta|}$, $m(T_{i_0}) = t \geq |\Gamma| - 3$, and we have an irredundant generating set $U_0 = \{u_1, u_2, \dotsc, u_t\}$ of $T_{i_0}$. 

Start with an empty collection of disjoint subsets of $\zeta = \{1, 2, \dotsc, |\Gamma|\}$, $\chi$. There exists a partition $\mathbf{Q_1} = \{R_1, R_2, \dotsc, R_{k_1}\}$ of $\zeta$ so that $U_0$ and $T_{i_0}$ are subsets of $\prod\limits_{i = 1}^{k_1} S_{R_i}$, and $T_{i_0}$ acts transitively on each $R_i$. By \cref{lm16}, $m(T_{i_0}) \leq |\Gamma| - k_1$, and so $k_1 \leq 3$. WLOG, assume that $|R_1| = \min\limits_{i \in \overline{1, k_1}}|R_i|$, and add $R_1 = L_1$ to $\chi$. Moreover, assume that $V_1 = \{u_1, \dotsc, u_r\}$ ($r \leq t$) so that $\pi_1(V_1)$ generates $\pi_1(U_0)$, where $\pi_1: \prod\limits_{i = 1}^{k_1} S_{R_i} \to S_{R_1}$ is the projection map defined in \cref{nt2}. Then there exists an irredundant set $U_1 = \{g_iu_i\}_{i > r} \subset S_{\zeta \setminus R_1}$ with $g_i$ generated by $\{u_1, \dotsc, u_r\}$. 

If $k_1 \geq 2$, there exists a partition $\mathbf{Q_2} = \{R'_1, R'_2, \dotsc, R'_{k_2}\}$ of $\zeta \setminus L_1$ so that $U_1 \subset \prod\limits_{i = 1}^{k_2} S_{R'_i}$, and the subgroup generated by $U_1$ acts transitively on each $R_i$. Assume again that $R'_1$ has the smallest size among $R'_i$, and add $R'_1 = L_2$ to $\chi$. If $k_2 \geq 2$, continue this process as before to get $V_2 \subset U_1$ that generates $\pi_1(U_1)$, where $\pi_1: \prod\limits_{i = 1}^{k_1} S_{R'_i} \to S_{R'_1}$ is the projection map defined in \cref{nt2}, and an irredundant set $U_2 \subset S_{\zeta \setminus (L_1 \cup L_2)}$. We will continue this process as long as $k_i \geq 2$ in step $i$.

The process cannot reach $L_4$ because otherwise, $m(T_{i_0}) \leq \sum\limits_{i = 1}^3 |V_i| + m(\genby{U_4}) \leq \sum\limits_{i = 1}^3 (|L_i| - 1) + |\zeta \setminus \cup_{i = 1}^3 L_i| - 1 = n - 4$. Contradiction. 

As a result, we get the partition $\chi = \{L_1, \dotsc, L_u \}$ of $\zeta$ so that $U_i \subset \prod_{j = i}^u S_{L_j}$ is an irredundant set, and the group generated by $U_i$  acts transitively on $L_i$.  Moreover, $m(T_{i_0}) \leq \sum\limits_{i = 1}^j m(\pi_i(U_i))$, where $\pi_i: \prod_{j = 1}^u S_{L_j} \to S_{L_i}$ is the projection map defined in \cref{nt2}. From the choice of $L_i$, we also get $|L_u| \geq \dfrac{|\Gamma|}{2^{u - 1}}$

We have $m(T_{i_0}) \leq \sum\limits_{i = 1}^u (|L_i| - 1) = |\Gamma| - u$. So $u \leq 3$.

If $u = 3$, then $|\Gamma| - 3 \leq m(T_{i_0}) \leq  \sum\limits_{i = 1}^3 m(\pi_i(U_i)) \leq \sum\limits_{i = 1}^3 (|L_i| - 1) = |\Gamma| - 3$. Hence the equality holds, and $\pi_i(U_i) = S_{L_i}$ and therefore, $S_{L_1} \times S_{L_2} \times S_{L_3} \subset T_{i_0}$, and so $T_{i_0}$ has $\mathbf{SP}$-property wrt the partition $\{P_0, L_1, L_2, L_3\}$, where $P_0$ is an empty set.

\vspace{2mm}
If $u = 2$, then $|L_2| \geq \dfrac{|\Gamma|}{2^{u - 1}} \geq \Psi(k - 1)$. By induction hypothesis on $k - 1$ for $U_2 \leq S_{L_2}$, we must have $A_{L_2} \subset U_2$ because otherwise, $m(\pi_2(U_2)) \leq  |L_2| - k$, and so $m(T_{i_0}) \leq |\Gamma| - k$. As a result, $m \leq |\Delta| - 1 + |\Gamma| - k + (|\Delta| - 1)(|\Gamma| - 2) = n - k + 1 - |\Delta| < n - k$. Contradiction. If $|L_1| \geq \Psi(k - 1)$, then we also have $A_{L_1} \leq \pi_1(U_1)$, and therefore $T_{i_0}$ satisfies $\mathbf{SP}$- property wrt the partition $\{P_0, L_1, L_2\}$, where $P_0$ is empty. Otherwise, $T_{i_0}$ satisfies $\mathbf{SP}$- property wrt the partition $\{P_0, L_2\}$, where $P_0 = L_1$ and has size less than $\Psi(k - 1)$.
\end{proof}

Now come back to the proof of \cref{thm7}. 

If $m(T_i) \leq |\Gamma| - 4 \ \forall i \in \overline{1, |\Gamma|}$, then $m \leq |\Delta| - 1 + |\Delta|(|\Gamma| - 4) = n - 3|\Delta| - 1 \leq n - 7 \leq n - k$. 

Otherwise, by \cref{lm17}, there is some $T_{i_0}$ that has $\mathbf{SP}$-property wrt some partition of $\mathbf{P} = \{P_0, P_1, \cdots, P_q\}$ of $\{1, 2, \dotsc, |\Gamma|\}$ ($q \leq 3$)

WLOG, suppose that $i_0 = 1$. Suppose that $T_1 = \{g_{l + 1},\dotsc, g_r\}$ such that $\{\pi_1(g_t)\}_{t = l+ 1}^r$ generates $G = \pi_1(K)$.

Suppose further that $T = \{h_ig_i,\ i > r,\ h_i \text{ generated by } T_1 \}$ is the set that fixes $\Gamma_1$ pointwise. 

By \cref{lm14}, there exists $T' \subset T$ with $|T'| \leq \omega(|P_0|, q)$ so that any subgroup $G$ with $\mathbf{K}$-property wrt $\mathbf{P}$ and $G$ that contains $\pi_2(T')$ also contains $\pi_2(T) \subset S_{\Gamma}$.
WLOG, suppose $T' = \{h_ig_i\}_{i = r+1}^s$, with $S \leq \omega(|P_0|, q)$.

There exist $\gamma$ generated by $\{g_1, \dotsc, g_l\}$ that moves block $1$ to block $2$. Then $\pi_2(\gamma g \gamma^{-1}) = x\pi_1(g)x^{-1}\ \forall g\in H $. By renaming elements in the second copy of $\Gamma$, $\Gamma_2$, we can assume that $x = 1$, and therefore, $\pi_2(\gamma g \gamma^{-1}) = \pi_1(g)$.

\vspace{2mm}
Now consider the group $M = N \cap \prod\limits_{i=1}^{|\Delta|} S_{\Gamma_i}$, where $N = \genby{\{g_1,\dotsc,g_s\}}$. 
\vspace{2mm}

Now suppose $\pi_i: \prod\limits_{i = 1}^{|\Delta|} S_{\Gamma_i} \to S_{\Gamma_i}$ is the projection map defined in \cref{nt2}.

\vspace{2mm}
From our previous observation, if $g = \pi_1(t_g) \in G$ for some $t_g \in T_1$, then $\pi_2(m_g) = \pi_2(\gamma t_g \gamma^{-1}) = \pi_1(t_g) = g$ ($m_g \in \prod\limits_{i = 1}^{|\Delta|}S_{\Gamma}$). Therefore, for $m \in M$ with $\pi_2(m) = t$, $m_g m m_g^{-1}$ also in $M$, and $\pi_2(m_gm m_g^{-1}) = gtg^{-1}$. Therefore $\pi_2(M)$ has $\mathbf{K}$-property wrt $\mathbf{P}$ and $G$. $\pi_2(M)$ contains $\pi_2(T')$, and, therefore, $\pi_2(M)$ must contain $\pi_2(T)$.

As a result, there exists $T_2 = \{h_tg_t, t > s\}$ that fixes both $\Gamma_1$ and $\Gamma_2$ pointwise, in which $h_t$ generated by $\{g_1,..,g_s\}\}$. 

By a similar argument as before, we get $m(H) = m = l + r-l + s- r + m-s \leq |\Delta| - 1 + |\Gamma| - 1 + \omega(|P_0|, q) + (|\Delta| - 2)(|\Gamma| - 2)$. 

Since $\omega(|P_0|, q) \leq \max\limits_{\substack{x \leq \Psi(k-1) \\ y \leq 3}} \omega(x, y) \leq |\Gamma| - k$, $m(H) \leq |\Delta| - 1 + |\Gamma| - 1 +|\Gamma| - k + (|\Delta| - 2)(|\Gamma| - 2) = n - |\Delta| - k + 2 \leq n - k$. Contradiction.

As a result, $m(H) \leq n -k$, and we finish the proof by induction on $k$.
\end{proof}

\begin{defi}
For a cycle $c = (u_1u_2 \dotsc u_s) \in S_n$ with $u_i \in \{1, 2, \dotsc, n\}$, we define $d(c) = s - 1$. For $x \in S_n$, we write the cycle decomposition of $x = c_1.c_2 \dotsc c_k$, where $c_i$ is one cycle, and then define $d(x) = \sum\limits_{i = 1}^k c_i$. 
\end{defi}
\begin{thm}\label{thm8}
Given a positive integer $k \leq 5$ and $n \in \mathbb{N}$ such that $n \geq \Psi(k + 2)$, where $\Psi$ is the function defined in \cref{thm7}. Suppose that $\{h_1, h_2,\dotsc, h_{n -k}\} $ is an irredundant generating set of $S_n$ of length $n - k$. Then $d(h_i) \leq k + 1$. Moreover, if $d(h_i) = k+1$, then $h_i$ must be an even permutation.
\end{thm}
\begin{proof}
Let $G_i$ be the proper subgroup of $S_n$ generated by $\{h_i\}_{i \neq j}$. Then, by \cref{thm7}, if $G_i \neq A_n$, then $G_i$ has to be intransitive, because otherwise, $G_i$ will have $m$ less than $n - k - 1$. Note that there is at most one $i$ so that $G_i = A_n$.

Now consider a graph $\Omega$ on the set of vertices $\Sigma = \{1, 2, \dotsc, n\}$. For each $i \in \overline{1, n - k}$ such that $G_i$ is intransitive, there exists $x$ and $y$ in $\Sigma$ so that $h_i$ maps $x$ to $y$ such that $x$ and $y$ are not in the same $G_i$-orbit. Such $x, y$ exists because if $x, y$ are in the same $G_i$-orbit whenever $h_i$ maps $x$ to $y$, then the orbits of $G$ are the same as the orbits of $G_i$. Contradiction since $G_i$ is transitive while $G$ is not. 

Now take $h_{i_1}$ with any $i_1 \in \overline{1, n - k}$, write the cycle decomposition of $h_{i_1}$, and represent each cycle $(u_1, u_2, \dotsc, u_k)$ in terms of product of $k-1$ transpositions $(u_{t-1}, u_t)$. We add edges that connect $u_{t -1}$ to $u_t \in \Sigma$ to the graph $\Omega$ to create new graph $\Omega'$.

Let $E(Gr)$ be the number of edges in a graph $Gr$.

If there is $h_{i_0}$ so that $G_{i_0} = 1$, then $E(\Omega) = n - k - 1$, and 

\[ E(\Omega') =  
\left\{
\begin{array}{ll}
n - k - 1 + d(h_{i_1}) - 1 \text{ if } i_1 \neq i_0 \\
n - k - 1 + d(h_{i_1}) \text{ if } i_1 = i_0 \\
\end{array}
\right.
\]

Note that in this case, $h_i$ with $i \neq i_0$ will be in $A_n$ and is therefore, an even permutation.

If there is no $h_{i_0}$ so that $G_{i_0} = 1$, then $E(\Omega) = n - k$, and $E(\Omega') = n - k + d(h_{i_1}) - 1$.

Therefore, $E(\Omega') \geq n - k + d(h_{i_1}) - 2$, and the equality holds only if $i_1 \neq i_0$, and $h_{i_1} \in A_n$.

Now we prove that $\Omega'$ is a forest. Suppose that there is a cycle in $\Omega'$. If Therefore, edges that can be in this cycle are edges $(u_{i-1}, u_i)$ that correponds to the cycle decomposition of $H$. But clearly, the edges are union of straight lines with $k-1$ edges $(u_{i-1}, u_i)$ that correspond to cycles $(u_1, u_2, \dotsc, u_k)$. Therefore, they cannot create a cycle. 

As a result, $\Omega'$ is a forest, and $E(\Omega') \leq n - 1$. Therefore, either $n - k + d(h_{i_1}) - 1 \leq n - 1$ or $n - k + d(h_{i_1}) - 2 \leq n - 1$, and $h_{i_1}  \in A_n$. In other word, $d(h_{i_1}) \leq k$ or $d(h_{i_1}) = k + 1$ and $h_{i_1}$ must be an even permutation.
\end{proof}

\begin{lem}\label{lm18}
Given $n \in \mathbb{N} \geq 3k + 3$. Suppose that $x \in S_n$ such that $d(x) \leq k$ or $x \in A_n$ and $d(x) = k + 1$, with $k$ odd. Then there exists an irredundant generating set of $S_n$ of length $n - k$ that contains $x$.
\end{lem}
\begin{proof}
By renaming elements of $\Sigma = \{1, 2,\dotsc, n\}$, we can assume that $x = \prod\limits_{i = 1}^m c_i$, where $c_i$ is the cycle $(d_i + 1, d_i + 2,\dotsc, d_{i+ 1})$, where $d_1 = 0$, and $d(x) = \sum\limits_{i = 1}^m (d_{i+1} - d_i - 1) = d_{m + 1} - m$. Here 

Suppose that $d(x) \leq k$.

Now consider the set $H$ that includes $x$, $m - 1$ transpositions $(d_{i + 1}, d_{i+1} + 1)$ for $i = \overline{1, m - 1}$, $l - d_{m + 1}$ other transpositions $\{(t, t+1)\}_{t = d_{m+1}}^{l - 1}$, and finally a cycle $(l, l + 1,\dotsc, n)$ with $n - k \leq l = n - k + d(x) - 1 \leq n - 1$.

Note that $d_{m + 1} = m + d(x) \leq 2 d(x) \leq 2(k + 1) \leq n - k - 1 < l$, so we have at least $1$ transposition of the form $\{(t, t+1)\}_{t = d_{m+1}}^{l - 1}$.

Then $H$ has $1 + m - 1 + l - d_{m + 1} + 1 = l + 1 - d(x) = n - k$ elements.

If we remove one transposition $(u, u + 1)$ in $H$ or the cycle $(l, l + 1,\dotsc, n)$, then the subgroup generated by the remaning elements cannot map $u$ to $u + 1$ or $l$ to $l + 1$. If we remove $x$ then it cannot even map $1$ to $2$. Therefore, $H$ is irredundant. 

Because $x$ and the first $m - 1$ transpositions $(d_{i + 1}, d_{i+1} + 1)$ generate a single cycle that includes all elements from $1$ to $d_{m + 1}$, and, therefore, together with the transposition $(d_{m + 1}, d_m)$ will generates the whole $S_{\{1, 2,.., d_{m + 1}\}}$. As a result, $H$ will generate $S_n$, and so $H$ an irredundant generating set of length $n - k$ of $S_n$ that contains $x$.

Now for even permutation $x$ with $d(x) = k + 1$, we consider the set $K$ that includes $x$, the transposition $(1, 2)$, $m - 1$ products of $2$ transpositions $(1, 2) \cdot (d_{i + 1}, d_{i+1} + 1)$ for $i = \overline{1, m - 1}$, $n - d_{m + 1} = n - m - k - 1$ other products of $2$ transpositions $\{(1, 2) \cdot (t, t+1)\}_{t = d_{m+1}}^{n - 1}$

By a similar argument, $K$ is also an irredundant generating set of length $n - k$ of $S_n$ that contains $x$, and we finish the proof for \cref{lm18}.
\end{proof}

\begin{cor}
Given a positive integer $1 \leq k \leq 5$ and a natural number $n$ so that $n \geq \max\{\Psi(k + 2), 3k + 3\}$, where $\Psi$ is the function defined in \cref{thm7}. Then $\iota_{n -k}(S_n) = \{x \in S_n: d(x) \leq k\} \cup \{ \text{even permutation } x: d(x) = k + 1,\ k \text{ odd}\}$. Here $\iota_m(G)$ is the set of elements in a finite group $G$ that lies in some irredundant generating set of length $m$ of $G$.
\end{cor}
\begin{proof}
This follows from \cref{thm8} and \cref{lm18}.
\end{proof}

\section{Dimension-like invariants of wreath product}
\subsection{Computing $m$ of wreath product of simple group and $2$-transitive permutation group}

\begin{lem}(Goursat's lemma for simple group)\label{lm19}
Let $S$ be a simple group, and $K$ be a subdirect product of $S^2$. Then either $K = S^2$ or $K = \{(s, \phi(s)), s \in S \}$, for some automorphism $\phi$ of $S$.
\end{lem}
\begin{lem}\label{lm20}
Let $S$ be a non-abelian simple group, and $K$ be a subdirect product of $S^n$. If the projection $\pi_{ij}: K \to S^2$ onto $i$th and $j$th coordinates is surjective for any $i \neq j$, then $K = S^n$.
\end{lem}
\begin{thm}(P. Hall)\label{thm9}
Let $S$ be a non-abelian simple group, and $s_1,\dotsc, s_n$ be $n$ elements in $S^k$. Then $k$ coordinate of $s_1\times s_2 \times\dotsc\times s_n$ generates $S^n$ if and only if $s_i$ generate $S$, and for any $i \neq j$, there is no automorphism of $S$ that maps $s_i$ component-wise into $s_j$.
\end{thm}

\begin{thm}\label{thm10}
Let $S$ be a non-abelian simple group, and $P \neq \mathbb{Z}_2$ be a $2$-transitive permutation group. Then $m(S) + m(P) \leq m(S \wr P) \leq m(S) + m(P) + 1$. Moreover, 
\begin{itemize}[leftmargin=0.2 in]
	\item If $S$ doesn't satisfy replacement property, then $m(S \wr P) = m(S) + m(P) +1$
	\item If $S$ satisfies replacement property, then $m(S \wr P) = m(S) + m(P)$
\end{itemize}
\end{thm}
\begin{proof}
For $K \subset S^n$ and $1 \leq i \neq j \leq n$, let $\pi_i(K) \subset S$ be the image of the projection from $K$ to the $i$th coordinate, and $\pi_{ij}(K) \subset S^2$ be the image of the projection from $K$ to the $i$th and $j$th coordinates.

Now assume that $P$ is a $2$-transitive subgroup of $S_n$ for $n > 2$, and $G = S \wr P = S^n \rtimes P$, where $P$ acts on $S^n$ by permuting $n$ copies of $S$.

First, suppose that $\{s_1,\dotsc, s_{m(S)}\}$ is an irredundant generating set of $S$, and $\{p_1, p_2,\dotsc, p_{m(P)}\}$ is an irredundant generating set of $P$. Let $H = \{s_1 \times (1_S)^{n-1}, s_2 \times (1_S)^{n-1},\dotsc, s_{m(S)} \times (1_S)^{n-1}, k_1, k_2,\dotsc, k_{m(P)} \}$, where $s_i \times (1_S)^{n-1} \in S^n$ and $k_i = (1, s_i) \in {1_{S^n}} \times P$ such that $s_i$ acts on $S^n$ by purely permuting copies of $S$. It is easy to see that $H$ is an irredundant generating set of $S \wr P$, and so $m(S \wr P)  \geq m(S) + m(P)$

Suppose that $m(G) = m(S \wr P) = m$, and take an irredundant generating set of length $m$, $H = \{h_1, h_2,\dotsc, h_m\}$. Applying \hyperref[lm1]{Whiston}'s lemma for $G$ and its normal subgroup $S^n$, we can assume that images of $H_1 = \{h_1, h_2,\dotsc, h_r\}$ in $G/S^n = P$ is an irredundant generating set of $P$, and $h_{r+1},\dotsc,h_m \in S^n$. It is obvious that $r \leq m(P)$ .

Let $G_1$ be the group generated by $H_1$. Now note that $S^n = G \cap S^n$ is generated by elements in the set $X = X_1 \cup X_2 = \{gh_ig^{-1}, i > r, g \in G_1\} \cup (G_1 \cap S^n) \subset S^n$. As a result, $\pi_1(X)$ generates $S$, and so some subset of the set $\pi_1(X)$ must be an irredundant generating set of $S$. Therefore, there exists $\Gamma \subset X$ such that $|\Gamma| \leq m(S)$, and $\pi_1(\Gamma)$ is an irredundant generating set of $S$. Consider all $h_i$ with $i > r$ such that some $gh_ig^{-1}$ is in $\Gamma$. Since $|\Gamma| \leq m(S)$, there are at most $m(S)$ such $h_i$'s. WLOG, we can assume that those $h_i$'s are $h_{r+1},\dotsc, h_l$ with $l-r \leq m(S)$. As a result, if we let $K_0 = \genby{h_1, \dotsc, h_l} \cap\ S^n$, then $\pi_1(K_0) = S$. 

\begin{lem}\label{lm21}
If $K = L \cap S^n$ is some subgroup of $S^n$ such that $H_1 \subset L \leq G$, and $\pi_1(K) = S$, then either $K = S^n$ or $K = \{(s, \phi_2(s),\dotsc, \phi_n(s)), s \in S \}$, where $\phi_i$ is some inner automorphism of $S$.
\end{lem}
\begin{proof}
First we show that such $K$ is a subdirect product of $S^n$. For $i \in \overline{2, n}$, there exists $\sigma_i \in G$ generated by $H_1$ that move $i$th copy of $S$ to the $1$st copy. Suppose that $\sigma_i = a.b$, where $a \in S^n$, and $b$ is a pure permutation of copies of $S$. Then, for any $k \in K$, $k' = \sigma_i^{-1}k\sigma_i$ is also in $K$, and so $\pi_1(a)^{-1}\pi_1(k)\pi_1(a) = \pi_i(k') \in \pi_i(K)$. Therefore $S = \pi_1(a)^{-1}\pi_1(K)\pi_1(a)^{-1} \subset \pi_i(K)\ \forall i \neq 1$. Hence, $K$ is a subdirect product of $S^n$. 

This means that $\pi_{ij}(K)$ is also a subdirect product of $S^2$. By \hyperref[lm19]{Goursat}'s lemma, either $\pi_{ij}(K) = S^2$ or $\pi_{ij}(K) = \{(s, \phi(s)), s \in S \}$ for some automorphism $\phi$ of $S$. 

Suppose that for all $i \neq j$, $\pi_{ij}(K) = \{(s, \phi(s)), s \in S \}$. If $\pi_{1i}(K) = \{(s, \phi_i(s)), s \in S \}$ for some automorphisms $\phi_i$ of $S$, then it is obvious that $K = \{(s, \phi_2(s),\dotsc, \phi_n(s)), s \in S \}$.

For any $i \neq j \in \overline{1, n}$, choose $k \in \overline{1, n}$ such that $i, j \neq k$. Note that $K$ can also be rewritten as $\{\phi_k^{-1}(s),\dotsc, \phi_k^{-1}\phi_{k-1}(s), s,\dotsc,\phi_k^{-1}\phi_n(s)), s\in S \}$ or $\{(\psi_1(s), \dotsc., \psi_{k-1}(s), s, \psi_{k+1}(s),\dotsc,\psi_n(s)), s \in S\}$ for some automorphisms $\psi_i$ of $S$. There exists $\sigma_{ij} \in G$ generated by $H_1$ that moves $j$th copy to $i$th copy of $S$ and doesn't move the $k$th copy because $P$ is $2$-transitive. For any $h \in K$, $h' = \sigma_{ij}^{-1}h\sigma_{ij} \in K$ with $\pi_k(h') = a_{ij}\pi_k(h)a_{ij}^{-1}$, and $\pi_j(h') = b_{ij}\pi_i(h)b_{ij}^{-1}$ for some $a = a_{ij}, b = b_{ij} \in S$ that only depends on $\sigma_{ij}$. Because $\pi_j(h') = \psi_j(\pi_k(h'))$, $b\pi_i(h)b^{-1} = \psi_j(a\pi_k(h)a^{-1}) = \psi_j(a)\psi_j(\pi_k(h))\psi_j(a)^{-1}$. Hence if we let $c_{ij} = b^{-1}\psi_j(a)$, then $\pi_i(h) = c_{ij}\psi_j(\pi_k(h))c_{ij}^{-1} = c_{ij}\pi_j(h)c_{ij}^{-1}\ \forall h \in K$. As a result, for any $i \neq j \in \overline{1, n}$, there exists $c_{ij} \in S$, $\pi_i(h) = c_{ij}\pi_j(h)c_{ij}^{-1}\ \forall h \in K$. Therefore, for $i > 1$, $\phi_i$ is an inner automorphism such that $\phi_i(s) = c_{i1}sc_{i1}^{-1}$.

Now suppose $\pi_{i_0j_0}(K) = S^2$ for some $i_0 \neq j_0 \leq n$ instead. Then because $P$ is $2$-transitive, by a similar argument as above, $(axa^{-1}, byb^{-1}) \in \pi_{ij}(K)$ for each $(x, y) \in \pi_{i_0j_0}(K) = S^2$. Hence, $\pi_{ij}(K) = S^2\ \forall i \neq j \leq n$. Then, by  \cref{lm20}, $K$ must be $S^n$.
\end{proof}

Now we come back to the proof of \cref{thm10}.

If $K_0 = S^n$, then $m = l =  r + l - r \leq m(P) + m(S)$. Otherwise, by \cref{lm21}, $K_0$ must be $\{(s, \phi_2(s),\dotsc, \phi_n(s)), s \in S \}$ for some automorphisms $\phi_i$. If $m > l$, then because of the irredundancy of $\{h_i\}$, $h_{l+1} \not \in K_0$. Then $K_1 = \genby{h_1,\dotsc,h_{l+1}} \cap\ S^n$ is also a subdirect product of $S_n$. If $K_1$ has the form $\{(s, \phi'_2(s),\dotsc, \phi'_n(s)), s \in S \}$, then because $K_0 \subset K_1$, $\phi_i(s) = \phi'_i(s)$ for each $s \in S$. Hence $K_1 = K_0$, and this is impossible since $h_{l+1} \in K_0$. Therefore, by \cref{lm21}, $K_1 = S^n$, and so $m = l + 1 = r + l - r +1 \leq m(P) + m(S) + 1$.

If $S$ doesn't satisfy the replacement property, then there exists an irredundant generating set $T = \{t_1, t_2,\dotsc, t_{m(S)}\}$, and a $t \in S$ such that $t$ cannot replace any $t_i$ in $T$. 

Consider the set $H = \{(t_1, t_1, \dotsc, t_1), \dotsc, (t_{m(S)}, t_{m(S)}, \dotsc, t_{m(S)}),$ \\ $(t, 1, \dotsc, 1),$ $p_1,\dotsc, p_{m(P)}\}$, where the first $m(S) +1$ elements are in $S^n$, and $p_i$ acts as pure permutations such that $\{\overline{p_i}\}$ is an irredundant generating set of $P$. If we remove some $(t_i, t_i, \dotsc, t_i)$, then we cannot even generate the first coordinate. Using this observation, it is easy to see that $H$ is irredundant. Because $P$ is transitive, $H$, in fact, generates $H' = H \cup \{u_i\}_{i=1}^n$ where $u_i = (1, 1,\dotsc, t,\dotsc, 1) \in S^n$ such that $t = \pi_i(u_i)$. By \hyperref[thm9]{P. Hall}'s theorem, $H' \cap S^n$ generates $S^n$, and therefore $H$ generates $G$.  As a result, $m(G) = m(P) + m(S) + 1$.

Now suppose that $S$ satisfies the replacement property. We will prove that $m(S \wr P) = m(S) + m(P)$. 

Assume by contradiction that $m(S \wr P) = m(S) + m(P) + 1$. Then $K_0$ must be $\{(s, \phi_2(s),\dotsc, \phi_n(s)), s \in S \}$ for some inner automorphisms $\phi_i$. Also, the set $\Gamma$ we considered must be $\{g_ih_ig_i^{-1},\ i \in \overline{r+1, l}, g_i \in G_1\}$, with $l-r = m(S)$ elements. Moreover, $\forall i \in \overline{1, n}$, the set of $i$th coordinates of these elements is an irredundant generating set of $S$. Therefore, WLOG, we can assume that the first coordinate of $h_{l+1} \neq 1_{S^n}$ is a non-identity element of $S$. Assume further that we can replace first coordinate of $g_{r+1}h_{r+1}g_{r+1}^{-1}$ by that of $h_{l+1}$ so that the first coordinates of elements in $\{g_ih_ig_i^{-1},\ i \in \overline{r+2, l}, g_i \in G_1\} \cup  \{h_{l+1}\}$ still generate $S$. Hence, if we let $K_1 = \genby{h_1,\dotsc, h_r, h_{r+2},\dotsc,h_{l+1}} \cap\ S^n$, then $\pi_1(K_1)$ is still $S$.  

By \cref{lm21}, we get $K_1 = \{(s, \phi'_2(s),\dotsc,\phi'_n(s)), s\in S \}$ for some inner automorphisms $\phi'_i$. 

Let $s_i = \pi_1(g_ih_ig_i^{-1})$ for $i \in \overline{r+1, l}$, and $s_{l+1} = \pi_1(h_{l+1})$. Assume also that $\forall s\in S, \ \phi_i(s) = c_isc_i^{-1}$, and $\phi'_i(s) = c'_i s {c'_i}^{-1}$. First, we have that $\{s_{r+1}, s_{r+2},\dotsc,s_l\}$ is an irredundant generating set of length $m(S)$ in $S$.  Suppose there is a proper subset $I$ of $J = \{ r+2,\dotsc, l+1\}$ such that $\genby{s_i}_{i \in I}$ generates $S$. Then $\pi_1(K_{I}) = S$, with $K_{I} = \genby{\{h_1, h_2,\dotsc, h_r\} \cup \{h_i\}_{i \in I}} \cap \ S^n$. Hence $K_{I} = K_1$, and for any $i_0 \in J\setminus I$, $h_{i_0} \in \genby{\{h_i\}_{i \neq i_0} }$. This is impossible because of the irredundancy of $\{h_i\}$. Therefore, $\{s_{r+2},\dotsc, s_{l+1}\}$ is also an irredundant generating set of $S$.

For $k > 1$, there is a $\sigma_k$ generated by $H_1$ that swap the $1$st copy and the $k$ copy. Hence, there exists $a_k$ and $b_k$ that depend only on $\sigma_k$ such that for any $x \in S$, and $y = \phi_k(x)$, $\sigma_k(x,\dotsc,y,\dotsc)\sigma_k^{-1} = (a_kya_k^{-1},\dotsc,b_kxb_k^{-1},\dotsc) \in K $. Let $a = a_k$, $b = b_k$, and $c = ac_k$. Now we have the following observation.

\begin{lem}\label{lm22}
Suppose that $T = \genby{H_1 \cup \{h_i\}_{i \in I_1}}$ is a subgroup of $K_0$ for some subset $I_1$ of $J_1 = \{r+1,\dotsc,l\}$. If there is some $h \in T$ such that $\pi_1(h) = x$, then for any $k \in \mathbb{Z} \geq 0$, there exists $h' \in T$ such that $\pi_1(h') = c^kxc^{-k}$.
\end{lem}
\begin{proof}
We will prove the statement by induction on $k \in \mathbb{Z} \geq 0$. 

For $k = 0$, the statement is trivial. Suppose the statement is true for $k - 1$ $(k \geq 1)$. We will prove that this is true for $k$.

By induction hypothesis, for any $h \in T$ such that $\pi_1(h) = x$, there exists some $h' \in S^n = (x',\dotsc, y', \dotsc)$ with $x' = c^{k-1}xc^{-k+1}$ and $y' = \phi_k(x') = c_kx'c_k^{-1}$ such that $h' \in T$. Therefore, $h'' = \sigma_kh'\sigma_k^{-1} = (ay'a^{-1},\dotsc,bx'b^{-1},\dotsc)$ is also in $T$. 

Note that $\pi_1(h'') = ay'a^{-1} = ac_kx'c_k^{-1}a^{-1} = cx'c^{-1} = cc^{k-1}xc^{-k+1}c^{-1} = c^{k}xc^{-k}$. Hence, we can finish the induction step here.
\end{proof}

Now suppose that $c \neq 1$. Because $\{s_{r+1},\dotsc,s_l\}$ is an irredundant generating set of length $m(S)$ in $S$, and $S$ satisfy replacement property, there exist $j_0$ such that $\genby{c, \{s_i\}_{r+1 \leq i \neq j_0 \leq l}} = S$. Let $K_2 = \genby{H_1 \cup \{h_i\}_{r+1 \leq i \neq j_0 \leq l}} \cap \ S^n$ is a subgroup of $S^n$. 

By \cref{lm22}, $M = \genby{\{c^ks_ic^{-k}\}_{r+1 \leq i \neq j_0 \leq l,\ k \in \mathbb{Z}}}$ must be a subgroup of $\pi_1(K_2)$. Because $S$ is simple, $S$ is exactly the set of elements $x$ such that $x$ can be represented as product of $s$ and $s_i$ in a way that the sum of exponents $u$ in $s^u$ term in this representation is $0$. As a result, $S$ is generated by generators of $M$, and so $\pi_1(K_2) = S = M$. Again by \cref{lm21}, we get $K_2 = \{(s, \phi_2''(s),\dotsc, \phi_n''(s)), s\in S\}$. As a result, $|K_2| = |K_0| = |S|$ and therefore, $h_{j_0} \in \genby{H_1 \cup \{h_i\}_{r+1 \leq i \neq j_0 \leq l}}$. This is impossible because of the irredundancy of $\{h_i\}$.

Therefore, $c = 1$ and $c_k = a_k^{-1}$. Similarly, $c'_k = a_k^{-1}$. As a result, $c_k = c'_k$, and $\phi_k = \phi'_k \ \forall k > 1$. Hence, $K_0 = K_1$. However, $h_{l + 1} \in K_0$, which is a subgroup of $\genby{\{h_i\}_{i \leq l}}$. This yields contradiction with the irredundancy of $\{h_i\}$. As a result, $m(S \wr P) = m(S) + m(P)$ as desired. 
\end{proof}

\begin{cor}
We have $m(A_n \wr A_n) = 2n - 4$ for all $n > 0$. In particular, $m(A_5 \wr A_5) = 6$.
\end{cor}
\begin{proof}
This follows from \cref{thm5} and \cref{thm10}
\end{proof}

\begin{rmk}
Sophie Le found that $M_{23}$, $J_1$, $J_2$, and $PSL_2(p)$ for $p = 17, 23, 29,$\\ $41, 47$ don't satisfy the replacement property, and she calculated $m$ of these groups. Furthermore, G.Frieden discovered that groups $PSL_2(p)$ with $p = 7$ and $p = 31$, which has $m = 4$, satisfy replacement property. As a result, we can replace $S$ in \cref{thm10} by $M_{23}$, $J_1$, $J_2$, or $PSL_2(p)$ with $p = 7, 17, 23, 29, 31, 41, 47$

For instance, thanks to Sophie, we can now calculate $m(G) = m(M_{23} \wr M_{23}) = m(M_{23}) + m(M_{23}) + 1 = 6 + 6 + 1 = 13$ (Here $M_{23}$ acts on $23$ copies of $M_{23}$, so the order of $G$ should be very big!)
\end{rmk}

\subsection{Computing $i$ of wreath product of strongly flat simple group and permutation group}

\begin{lem}\label{lm23}
Suppose that $H \subset \prod\limits_{i = 1}^{k} S_i$, where $S_i \cong S$, which is a strongly flat group. Moreover, assume that $m = m(S) = i(S) \geq 3$, and $\pi_i(\genby{H}) \neq S\ \forall i \in \overline{1, k}$. Then there exists a subset $H' \subset H$ with $|H'| \leq (m - 1)k$ so that the group generated by $H'$, $\genby{H'}$, is exactly $\genby{H}$. Here $\pi_i: \prod\limits_{j = 1}^{k} S_j \to S_i$ is the projection map defined in \cref{nt2}
\end{lem}
\begin{proof}
We will prove by induction on $k$. For $k = 1$, it is obvious from the definition of $m = m(S) = i(S)$.

Suppose that the lemma is true for $k -1$. We will prove that the lemma is also true for $k$ ($k \geq 2$).

Choose the smallest subset $K$ of $\pi_1(H)$ that generates $\pi_1(\genby{H}) \neq S$. Then $K$ must be irredundant, and, by definition of $i(S) = m(S)$, for strongly flat $S$, $|K| \leq m - 1$. Suppose that $K = \{\pi_1(k), k \in K_1 \subset H \}$ for some subset $K_1 \subset H$ with $|K_1| \leq m -1$.

Therefore, for every $h \in H$, there exists $k_h$ generated by $K_1$ so that $k_hh \in \prod\limits_{i = 2}^k S_i$. 

Consider $H' = \{k_hh, h \in H \} \subset \prod\limits_{i = 2}^k S_i \cap \genby{H}$. Hence, $\pi_i\genby{H'} \subset \pi_i(\genby{H}) \neq S\ \forall i = \overline{2, k}$. Using induction hypothesis on $H' = \{k_hh, h \in H \} \subset \prod\limits_{i = 2}^k S_i$, there exists some $K' \subset H'$ with $|K'| \leq (m - 1)(k - 1)$ so that $\genby{K'} = \genby{H'}$. Assume that $K' = \{k_hh, h \in K_2 \}$ for some $K_2 \subset H$, $|K_2| \leq (k - 1)(m - 1)$. 

As a result, $\genby{K_1 \cup K_2} = \genby{H}$, and $|K_1 \cup K_2| \leq |K_1| + |K_2| \leq m-1 + (k -1)(m - 1) = k(m -1)$. So we finish induction step here, and therefore finish the proof for \cref{lm23}
\end{proof}
\begin{thm}\label{thm11}
For a simple strongly flat group $S$ with $m = m(S) = i(S)$, and permutation group $P$ of degree $n = deg(P)$, $i(S \wr P) = i(S).deg(P) = m.n$
\end{thm}
\begin{proof}
Suppose there is an irredundant set $\{h_1,\dotsc,h_{i(G)}\}$ in $G = S\wr P$. Let $H \leq G$ be the subgroup generated by these $h_i$. Let $N = S^n = \prod\limits_{i = 1}^n S_i \trianglelefteq G$ ($S_i \cong S$).

Using \hyperref[lm1]{Whiston}'s lemma for $H$ and its normal subgroup $(H \cap N) \trianglelefteq H$, we can assume that for some $g_1,\dotsc,g_l \in S^n$, and the images of $\{h_1, h_2,\dotsc, h_k\}$ under the projection map: $H \to H/(H \cap N)$, $\overline{h_1}, \overline{h_2},\dotsc, \overline{h_l}$ is an irredundant generating set of $K = H/(H \cap N) \cong HN/N \leq G/N \cong P \leq S_n$ ($i(G) = l + k$). 

There exists a partition $\mathbf{P} = \{X_1, X_2,\dotsc, X_p\}\ (p \geq 1)$ of $\{1, 2,\dotsc, n\}$ so that $K \leq \prod\limits_{j = 1}^pS_{X_j}$, and acts transitively on each of $X_i$. WLOG, assume that $X_i = \{d_{i - 1} + 1,\dotsc, d_i\}$, with $d_0 = 1$ and $d_p = n$

\begin{lem}\label{lm24}
For $1 \leq q \leq p + 1$ that, after reordering $g_i$, there exists $g_1, \dotsc, g_s$ with $s \leq \sum\limits_{i = 1}^{q-1} \max\{m + |X_i| - 1, (m - 1)|X_i|\}$ such that $\forall t \in \overline{1, l}$, $x_tg_t \in \prod\limits_{i = d_{q-1} + 1}^n S_i$ for some $x_t$ generated by $g_1,\dotsc, g_s$ and $h_1, h_2,\dotsc, h_k$ (Here $\prod\limits_{i = d_p + 1}^nS_i = \{1\}$).
\end{lem} 

\begin{proof}
We will prove \cref{lm24} by induction on $q$. For $q = 1$, we don't need to prove anything. Assume that the statement is true for $q$. We will prove that it is true for $q + 1$ ($1 \leq q \leq p$).

By induction hypothesis on $q$, there exists $g_1, g_2,\dotsc,g_s'$ with \\ $s' \leq \sum\limits_{i = 1}^{q-1} \max\{m + |X_i| -1, (m-1)|X_i| \}$ so that $T = \{x_ig_i\}_{i = 1}^l \subset \prod\limits_{i = d_{q-1} + 1}^n S_{i}$, with some $x_i$ generated by $g_1,\dotsc,g_{s'}$ and $h_1,\dotsc,h_k$. 

Now let $\pi_i : S^n = \prod\limits_{i = 1}^n S_i \to S$ be the projection map from $S^n$ to its $i$th factor, and $\pi_{i_1, i_2, \dotsc, i_k}: S^n \to \prod\limits_{j = 1}^k S_{i_j}$ be the projection map from $S^n$ to some of its factors (defined in \cref{nt2}). We have two following cases to consider: 

\vspace{3mm}
$\mathbf{Case\ 1}$: $\pi_i (\genby{T}) \neq S \ \forall i = \overline{d_{q-1} + 1, d_q}$. By \cref{lm23} for $\pi_{\{d_{q-1} + 1,\dotsc,d_q\}}(\genby{T})$, there exists $T' \subset T$ with $|T'| \leq (m - 1)(d_q - d_{q-1}) = (m-1)|X_q|$ so that $\pi_{\{d_{q-1} + 1,\dotsc,d_q\}}(\genby{T'})$ = $\pi_{\{d_{q-1} + 1,\dotsc,d_q\}}(\genby{T})$. Suppose that $T' = \{x_ig_i, i \in I\}$ for some $|I| \leq (m-1)|X_q|$. Hence, there exists $z_i$ generated by $\{x_ig_i\}_{i \in I}$ and, therefore, by $\{g_i\}$ with $i \in \{1, 2,\dotsc, s\} \cup I$ and by $\{h_1,\dotsc, h_k\}$ so that $x'_i g_i = z_ix_ig_i \in \prod\limits_{i = d_q + 1}^ n S_i \ \forall i \in \overline{1, l}$. By reordering $g_i$, we can assume that $I = \{s+1,\dotsc, s'\}$ with $s'  \leq s + (m-1)|X_q| \leq \sum\limits_{i = 1}^{q-1} \max\{m + |X_i| - 1, (m - 1)|X_i|\} + (m -1)|X_q| \leq \sum\limits_{i = 1}^{q-1} \max\{m + |X_i| - 1, (m - 1)|X_i|\}$, and because, $x'_i$ is generated by $\{g_1,\dotsc, g_{s'}\}$ and $\{h_1, h_2,\dotsc, h_k\}$, we can finish induction step here. 

\vspace{3mm}
$\mathbf{Case\ 2}$: If $\pi_i (\genby{T})  = S$ for some $i \in \overline{d_{q-1} + 1, d_q}$. WLOG, we can assume that $i = d_{q- 1} + 1 = d'$. Moreover, after reordering $q_i$, assume that images of $\{x_ig_i\}_{i = s' + 1}^{s''}$ under the map $\pi_{d'}$, which are $\pi_{d'}(x_{s'}g_{s'}), \dotsc, \pi_{d'}(x_{s''}g_{s''})$, generates $S$. ($s'' - s' \leq m = m(S) = i(S)$). As a result, there exists $y_i$ generated by $\{g_1,\dotsc, g_{s''}\}$ and $\{h_1,\dotsc, h_k\}$ so that $\{y_ig_i\}_{i = 1}^l \subset \prod\limits_{i = d_{q-1} + 2}^n S_i$

\begin{lem}\label{lm25}
For any $0 \leq j \leq |X_i| - 1$, after reordering $g_j$, there exists $z_i$ generated by $\{g_1, g_2,\dotsc, g_{j +s''}\}$ and $\{h_1, h_2,\dotsc, h_k\}$ so that $\{z_ig_i\}_{i = 1}^l \subset \prod\limits_{i = d' + 1 + j}^nS_i$
\end{lem}
\begin{proof}
We will prove \cref{lm25} by induction on $j$. For $j = 0$, this is true obviously. Suppose that this is true for $j - 1$. Then there exists $z_i$ generated by $\{g_1, g_2,\dotsc, g_{j - 1 + s''}\}$ and $\{h_1, h_2,\dotsc, h_k\}$ so that $\{z_ig_i\}_{i = 1}^l \subset \prod\limits_{i = d' + j}^n S_i$. We will prove that the statement is also true for $j$.

If $\{z_ig_i\}_{i = s'' + j}^l \subset \prod\limits_{i = d' + 1 + j}^n S_i$, then we are done already. Otherwise, there will be some $g_{i_0}$ with $s'' + j \leq i_0 \leq l$ so that $\pi_{d' + j}(z_{i_0}g_{i_0}) \neq 1_S$, and, by reordering $g_i$, we can assume that $g_{i_0} = g_{s'' + j}$ with $\pi_{d' + j}(z_{s'' + j}g_{s'' + j}) = u_0$

$K$ is transitive on $X_q = \{d_{q-1} + 1, \dotsc, d_q\}$ so there exists $\gamma$ generated by $\{h_i\}$ that moves $d_{q-1} + 1 + j = (d' + j)$th copy of $S$ to $d_{q- 1} + 1 = d'$th copy of $S$. Then there exists $x \in S$ depending on $\gamma$ so that $\forall k \in \prod\limits_{i = 1}^n S_i,\ \pi_{d' + j}(k') = \pi_{d' + j}(\gamma^{-1}k\gamma)$ = $x\pi_{d'}(k)x^{-1}$. Since $\pi_{d'}(x_{s'}g_{s'}), \cdots, \pi_{d'}(x_{s''}g_{s''})$ generate $S$, and $xSx^{-1} = S$, there exists $k' \in \prod\limits_{i = 1}^n S_i$ generated by $\{h_1, \dotsc, h_k\}$ and $g_1,\dotsc, g_{s''}$ so that $\pi_{d' + j}(k') = v$ for any $v \in S$.

Then, for any $v \in S$, $\pi_{d' + j}(k'z_{i_0}g_{i_0}k'^{-1}) = vu_0v^{-1}$ ($i_0 = s'' + j$) with some $k'z_{i_0}g_{i_0}k'^{-1} \in \prod\limits_{i = d' + j}^n S_i$ generated by $\{g_1,\dotsc, g_{s'' + j}\}$ and $\{h_1,\dotsc, h_k\}$. Because $S$ is simple, $S$ is generated by $\{vu_0v^{-1}, v\in S\}$. As a result, for any $u \in S$, there exists $k'' \in \prod\limits_{i = d' + j}^n S_i$ generated by $\{g_1,\dotsc, g_{s'' + j}\}$ and $\{h_1,\dotsc, h_k\}$ so that $\pi_{d' + j}(k'') = u$, and so for all $i \in \overline{1, l}$, there exists $z''_i \in \prod\limits_{i = d' + j}^n S_i$ generated by $\{g_1,\dotsc, g_{s'' + j}\}$ and $\{h_1,\dotsc, h_k\}$ so that $\pi_{d' + j}(z''_i) = \pi_{d' + j}(z_ig_i)^{-1}$. Hence, $\{z''_iz_ig_i\} \subset \prod\limits_{i = d' + j + 1}^n S_i$, and we finish induction step here.
\end{proof}

Applying \cref{lm25} for $j = d_q - d_{q - 1} - 1 = |X_q| - 1$, then we finish the induction step for \cref{lm24}

\end{proof}
Now by \cref{lm24}, we must have $l \leq \sum\limits_{i = 1}^p \max\{m + |X_i| -1, (m-1)|X_i|\}$. Also, by definition of partition $\mathbf{P}$, we must have $k \leq m(K) \leq n - p$.

As a result, 
\begin{equation*} 
\begin{split}
i(G) & \leq n - p + \sum\limits_{i = 1}^p \max\{m + |X_i| - 1, (m - 1)|X_i|\} \\
& = n - p + \sum\limits_{i: |X_i| = 1}m + \sum\limits_{i: |X_i| > 1}(m- 1)|X_i| \\
& = n - p + \sum\limits_{i: |X_i| = 1}m + (m - 1)(\sum\limits_{i: |X_i| > 1}|X_i|) \\
& = n - p + tm + (m - 1)(n - t) \ (t = |\{i: |X_i| = 1\}| \leq p) \\
& = n - p + tm + mn - n - tm + t \\
& = mn + t - p \leq mn
\end{split}
\end{equation*}
	
Now take any irredundant generating set $\{s_1,\dotsc, s_m\}$ of $S$.

Consider the set $H = \{(1, 1,\dotsc, s_i, 1,\dotsc, 1)\} \subset S^n$, where $s_i$ can be in any $j$th positions for $1 \leq j \leq n$, and $i$ can be any number between $1$ and $m$. It is easy to see that $H$ is irredundant, and because $S^n \leq S \wr P$, $i(S\wr P) \geq |H| = mn$.

As a result, $i(S \wr P) = mn$.
\end{proof}

\begin{cor}
$i(A_n \wr A_n) = n^2 - 2n$ for $n \geq 9$.
\end{cor}
\begin{proof}
This follows from \cref{thm11} and \cref{cor1}.
\end{proof}

\section{Acknowledgment}

I would like to thank Professor Keith Dennis and Ravi Fernando for their great guidance and assistance, and Summer Math Foundation Fellowship and my donor for providing me with financial support to participate in the summer research. I also thank the SPUR program and Cornell's math department for organizing this summer research program.


\newpage
\begin{thebibliography}{9}
\bibitem{Whiston}\hypertarget{Whis}
Julius Whiston, Maximal Independent Generating Sets of the Symmetric Group, \textit{Journal of Algebra}, \textbf{232}, 255-268 (2000).
\bibitem{CC}
Peter J. Cameron and Philippe Cara, Independent generating sets and geometries for symmetric groups, \textit{Journal of Algebra}, \textbf{258}, 641-650 (2002).
\bibitem{Ravi}
Ravi Fernando, On an Inequality of Dimension-like Invariants for Finite Groups, 2015. \href{https://arxiv.org/pdf/1502.00360.pdf}{https://arxiv.org/pdf/1502.00360.pdf}
\bibitem{Wilson} \hypertarget{Wil}{}
Robert A. Wilson, \textit{The O’Nan–Scott Theorem}, 2005. \par \href{http://www.maths.qmul.ac.uk/~raw/talks_files/ONStalk.pdf}{http://www.maths.qmul.ac.uk/$\sim$ raw/talks\_files/ONStalk.pdf}
\bibitem{Gabe}
G. Frieden, Computing $m(S \wr Z_2)$, REU at Cornell 2011.
\end{thebibliography}
\end{document}